\documentclass[11pt,leqno]{article}
\usepackage[a4paper, textwidth=14cm]{geometry}
\usepackage{amssymb,amsmath}
\usepackage{graphicx}
\usepackage{color}
\usepackage{leftidx,tensor}

\usepackage{accents}

\newtheorem{theorem}{Theorem}[section]
\newtheorem{lemma}[theorem]{Lemma}
\newtheorem{remark}[theorem]{Remark}

\newtheorem{definition}[theorem]{Definition}

\newenvironment{proof}
{\begin{trivlist}
		\item[\hskip\labelsep{\bf Proof:}\quad]}
	{\hspace*{\fill}\rule{2mm}{2mm}
\end{trivlist}}

\newcommand{\bea}{\begin{eqnarray*}}
\newcommand{\eea}{\end{eqnarray*}}
\newcommand{\be}{\begin{equation}}
\newcommand{\ee}{\end{equation}}

\newcommand{\cll}{{\cal L}}

\newcommand{\clu}{{\cal U}}

\newcommand{\RRM}{\mathbb{R}}

\newcommand{\KKM}{\mathbb{K}}

\renewcommand{\Re}{\mathop{\rm Re}}

\newcommand\norm[1]{\left\lVert#1\right\rVert}
\usepackage{amssymb}

\newenvironment{keywords}%
   {\begin{trivlist}\item[]{\bfseries\sffamily Keywords:}\ }
   {\end{trivlist}}

\newenvironment{MSC}%
   {\begin{trivlist}\item[]{\bfseries\sffamily MSC:}\ }
   {\end{trivlist}}

\numberwithin{equation}{section}

\begin{document}

\title{Loss of embeddedness for the one-phase quasistationary Stefan problem in 2D}
\author{Friedrich Lippoth \\ \small{Lecturer at Leibniz University of Applied Sciences Hannover,} \\ \small{Expo Plaza 11, D-30599 Hannover, Germany} \\
\small{\texttt{friedrich.lippoth@leibniz-fh.de}} \\
}
\date{}
\maketitle

\begin{abstract}
We provide an example for a smooth and embedded initial state that looses embeddedness in finite time when evolving according to the quasistationary Stefan problem with Gibbs-Thomson correction and kinetic undercooling in 2D.
\end{abstract}


\begin{MSC}
80A22, 35R37, 53C42, 53E10 
\end{MSC}

\section{Introduction.}

\vspace{5mm}

The Stefan problem \cite{S1, S2} is an intensively studied representative of a class of moving boundary problems arising from a variety of models in continuum mechanics, other fields of physics as well as in the life sciences. The model describes phase transitions. (As an example one may think of the transition water $\leftrightarrow$ ice.) In our setting, the bulk (say liquid) phase is separated from the solid phase by a sharp interface. Unknowns are the temperature distribution and the shape of the interface. The dissertation \cite{Kn} and the articles \cite{PSZ, SF} (in addition to their respective contributions) outline the underlying physics and contain a broad variety of references as well, including approaches from different mathematical schools. 

\vspace{5mm}

One of the standard techniques for a rigorous mathematical treatment of moving boundary problems like the Stefan problem is the direct mapping method, i.e. transformation to a fixed reference domain / interface by a family of time-dependent diffeomorphisms (called the Hanzawa diffeomorphisms), and to apply methods from functional analysis to the resulting evolution problem. Typically, using this method, one proves local existence, uniqueness and regularity of classical / strong solutions or attractivity properties of the manifold of steady states.

A major difficulty is the understanding of the development of singularities in the moving boundary as it is inherent in the transformation process that the evolving geometry can be investigated only to the point where it touches the boundary of a tubular neighborhood of the (embedded) initial interface. A first (somehow naive) approach to deal with this difficulty can be found in \cite{L}. In this paper we attempt to overcome this limitation.

\vspace{5mm}

We rigorously prove for a quasistationary one-phase version of the model in 2D, that a smooth and simple initial geometry can loose its embeddedness in finite time. The idea is to (continuously) extend the problem to a (physically meaningless) case of immersed curves. To maintain a (mathematically) well defined bulk phase, we cut off overlappings of the boundary and use 'appropriately weak' regularity classes for the temperature in the resulting Lipschitz domains (sections 2 / 3). This puts us in the position to show that a suitable initial interface with touching points (evolving according to the generalized flow) must develop a true overlap in finite time, and, in a second step, extrapolate this behaviour to neighboring embedded initial configurations by means of a (semi-flow type) continuity argument inspired by \cite{MS}. When restricted to embedded geometries, the generalized problem is the classical problem, and the transformation process we apply resembles the standard procedure scetched above. 

\vspace{5mm}

A core element of our construction is the specification of an initial state that directs the system (in its early stages) towards a topological change (section 4). This task is non-trivial, since we are considering the quasistationary model and the initial data of the flow need to satisfy a compatibility condition.

\vspace{5mm}

Physically, our result refers to a situation where a hot liquid is given in a cold solid that has a 'small straight bridge' in the liquid. This bridge is then shown to be partially melted. The result seems natural and may be expected, but, to our best knowledge, has not been proven so far. (Generally, even though the Stefan problem has been studied for more than a century, only very few rigorous qualitative results are available, see \cite{PSZ} for a summary.) Finally, we remark that we are not aware of obvious limitations that could inhibit the transfer of our findings to higher space dimensions, particularly to 3D. 
 
\vspace{15mm}
 
\begin{minipage}{0.35\textwidth}
\includegraphics[width=4cm,height=3cm]{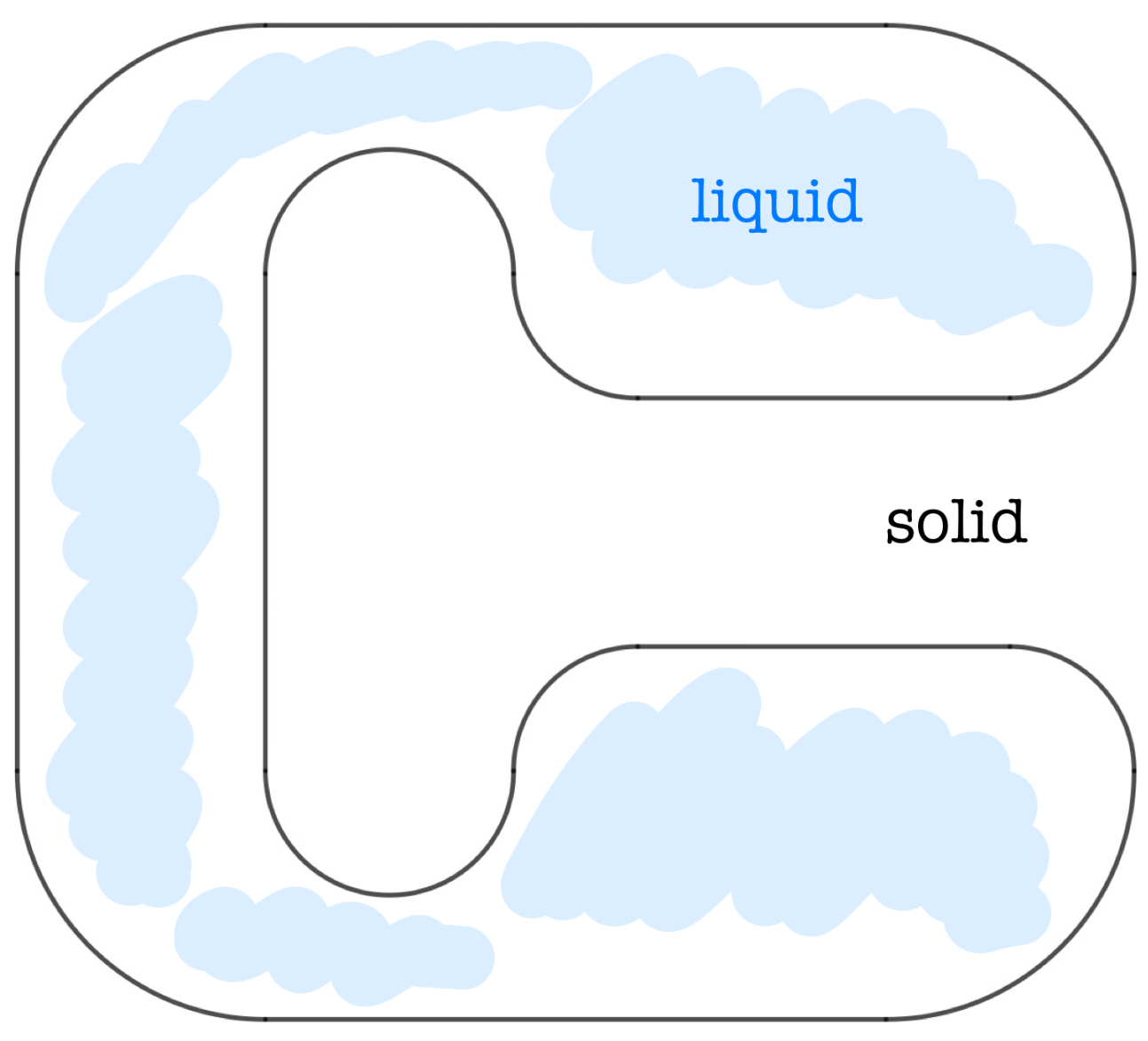} 
\end{minipage}
\begin{minipage}{0.2\textwidth}
\[
\Longrightarrow
\] 
\end{minipage}
\begin{minipage}{0.45\textwidth}
\includegraphics[width=4cm,height=3cm]{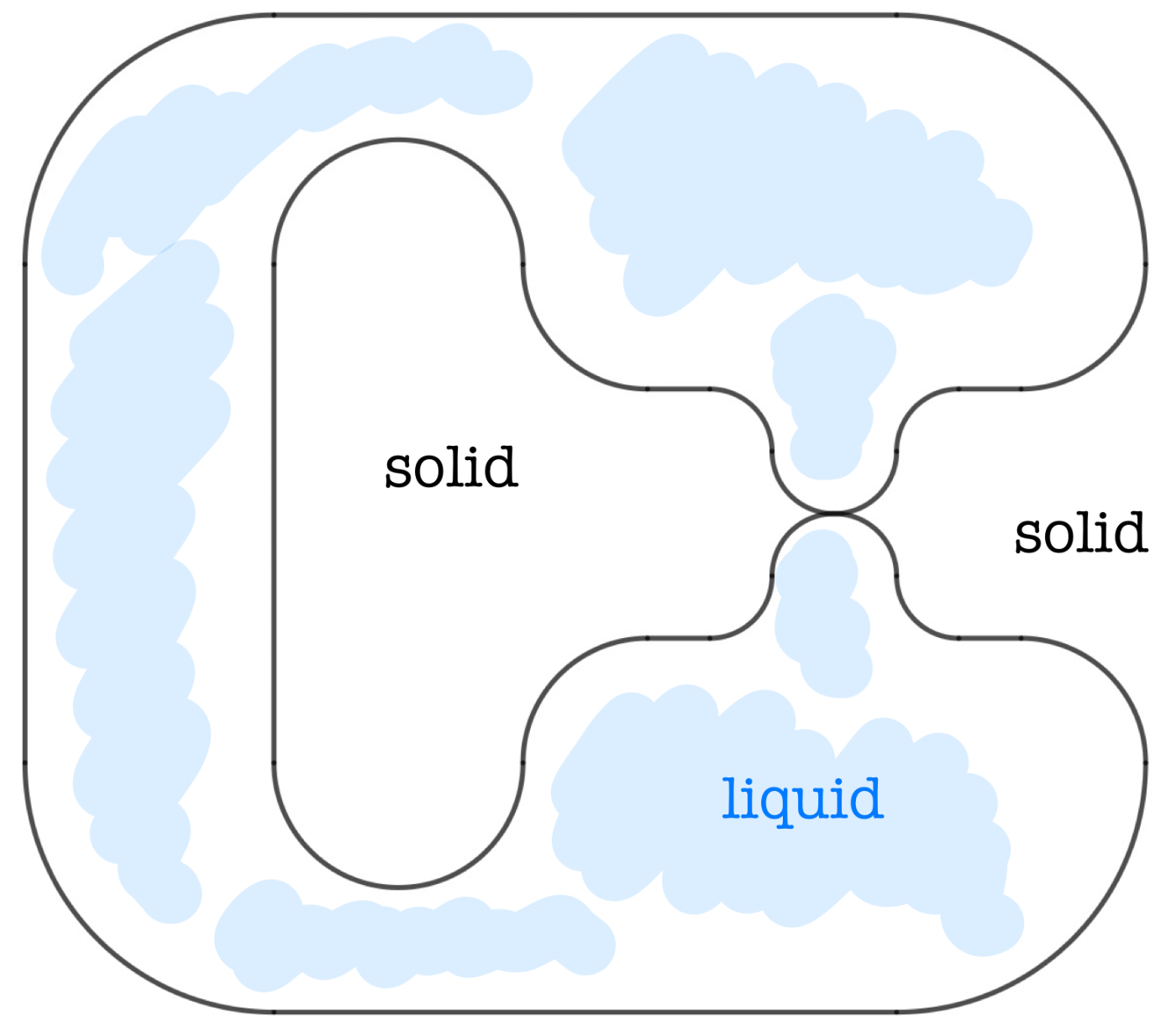}
\end{minipage} 
 
\newpage 

\subsection{ The classical quasistationary one-phase Stefan problem in 2D (including the effects of surface tension and kinetic undercooling)...}

\vspace{5mm}

... reads as follows: We search for a family of embeddings $F(t,\cdot): S = \RRM / 2\pi\mathbb{Z} \rightarrow \RRM^2$ satisfying the evolution law 
\[
(\dot{F} | N_F) = H_F - u_F, 
\]
which is explained below. (Physical constants are inadmissibly normalized to $1$ for simplicity since this has no qualitative effect on our analysis.)

\vspace{5mm}

Let $\gamma_F \subset \RRM^2$ be the simple curve with curvature $H_{\gamma_F}$ (positive for positively oriented spheres) and inner unit normal $N_{\gamma_F}$ parameterized by $F(t,\cdot): S \rightarrow \RRM^2$, and let $\Omega_F \subset \RRM^2$ be the domain enclosd by $\gamma_F$. ( We use the notation 
\[
H_F := H_{\gamma_F} \circ F, \qquad \qquad N_F := N_{\gamma_F} \circ F. \; \mbox{)}
\]
Then $u_{\gamma_F}(t,\cdot)$ is the unique solution of   
\begin{equation}\label{bulk1}
\left\{ \begin{array}{lcll} \Delta u & = & 0 & \mbox{ in } \Omega_F, \\ u - \partial_{N_{\gamma_F}} u & = & H_{\gamma_F} &  \mbox{ on } \gamma_F \end{array} \right.
\end{equation}
and $u_F := u_{\gamma_F}(t,\cdot)|_{\gamma_F} \circ F$. (The unknown $u$ is temperature.) 

\vspace{5mm}

To motivate our ansatz outlined in sections 2 / 3, let us briefly scetch how the Stefan problem can be solved in an embedded setting: 

\vspace{5mm}

(In the following we consider functions $F = F(s)$, $\rho = \rho(s)$ of one variable $s \in S$ (i.e. independent of time) until explicitly stated otherwise.)

\vspace{5mm}

A weak solution $u_{\gamma_F} \in W_2^1(\Omega_F)$ can be obtained by solving  
\begin{equation}\label{ws1}
\int_{\Omega_F} \nabla u \nabla \psi + \int_{\gamma_F} u \psi = \int_{\gamma_F} H_{\gamma_F} \psi \qquad \qquad ( \psi \in W_2^1(\Omega_F,\RRM) )
\end{equation}
by a standard indirect argument involving Lax-Milgram and compactness. Let us assume that the curve $\gamma_F$ is a graph in normal direction over some smooth simple initial curve $\gamma_0 = F_0[S]$ (enclosing the domain $\Omega_0$), i.e $\gamma_F$ is parameterized by  
\begin{equation}\label{rho}
F(s) = F_0(s) - \rho(s) N_0(s), \qquad \qquad (s \in S, \quad \rho: S \rightarrow \mathbb{R}).
\end{equation}
($N_0(s) = (N_{\gamma_0} \circ F_0)(s)$ is the inner unit normal vector of $\gamma_0$ at $p = F_0(s)$.) Observe that for given $F$, $\rho = (F_0 - F | N_0)$. Hence, there is a smooth one to one correspondence between the embedding $F$ and the distance function $\rho$. We refer to this correspondence by writing $F = F(\rho)$, $\rho = \rho(F)$.

We denote by $\mathcal{N}_0$ a tubular neighborhood of $\gamma_0$ and by $\Phi_0^{-1} = (P_0, \Lambda_0)$ the inverse of the smooth diffeomorphism  
\[
\Phi_0(p,\delta) = p - \delta N_{\gamma_0}(p): \quad \gamma_0 \times (-a_0,a_0) \rightarrow \mathcal{N}_0. 
\]
($P_0$ denotes metric projection onto, and $\Lambda_0$ is signed distance w.r.t $\gamma_0$.) The diffeomorphism 
\[
\theta_F(p) = p - \rho(p) N_{\gamma_0}(p): \quad \gamma_0 \rightarrow \gamma_F \qquad \qquad (p = F_0(s))
\]
can be extended to a diffeomorphism $\Omega_0 \rightarrow \Omega_F$, $\bar\Omega_0 \rightarrow \bar\Omega_F$, $\gamma_0 \rightarrow \gamma_F$, $\mathbb{R}^2 \rightarrow \mathbb{R}^2$ via 
\[
\Theta_F(p) := 
\left\{ \begin{array}{ll} P_0(p) - \left[ \Lambda_0(p) + (\varphi_0 \circ \Lambda_0)(p) \rho(P_0(p)) \right] N_{\gamma_0}(P_0(p)), & p \in \mathcal{N}_0, \\ p, & p \in \mathbb{R}^2 \setminus \mathcal{N}_0, \end{array} \right.
\]
where $\varphi_0 \in C^\infty(\mathbb{R},[0,1])$ satisfies for example  
\[
\varphi_0(r) = 
\left\{ \begin{array}{ll} 1, & |r| \leq a_0/4, \\ 0, & |r| \geq 3a_0/4 \end{array} \right.
\] 
and $|\varphi_0'| \leq 3/a_0$. 

\vspace{5mm}

(We have now that $F = \Theta_F|_{\gamma_0} \circ F_0$. We also have denoted the function
\[
p \mapsto \rho(s): \gamma_0 \rightarrow \RRM \qquad \qquad (p = F_0(s))
\]
by the same symbol $\rho$. Its differential $T_p \gamma_0 \rightarrow \RRM$ will be denoted by $d_p \rho$. It is important to consider only functions s.t. $|\rho| \leq a_0/4$, cf. section 2.) 

\vspace{5mm}

Let   
\begin{eqnarray*}
a_F(v,\varphi) & := & \int_{\Omega_0} |\mbox{det} D \Theta_F| ( \nabla v | \nabla \varphi )_F + \int_{S}|D \Theta_F(F_0) F_0'| (v \varphi)(F_0) \\
b_F(\varphi) & := & \int_{S} |D \Theta_F(F_0) F_0'| H_F \varphi(F_0), \\ 
\end{eqnarray*}
where 
\[
\left( \nabla v | \nabla \varphi \right)_F = \left( \; \left( D\Theta_F^{-1} \right)^T \left( \Theta_F \right) \nabla v \; | \;  \left( D\Theta_F^{-1} \right)^T \left( \Theta_F \right) \nabla \varphi \; \right), 
\]
and note that the forms $a_F: W_2^1(\Omega_0,\RRM) \times W_2^1(\Omega_0,\RRM) \rightarrow \RRM$ are uniformly coercive and continuous w.r.t. $F$ sufficiently $W_2^2(S,\RRM^2)$-close to $F_0$. We name $v_{\gamma_F}$ the unique solution of the operator equation $a_F(v,\varphi) = b_F(\varphi)$, $\varphi \in W_2^1(\Omega_0,\RRM)$. The function 
\[
u_{\gamma_F} := v_{\gamma_F} \circ \Theta_F^{-1} \in W_2^1(\Omega_F,\RRM)
\]
is the unique weak solution (\ref{ws1}). 

\newpage

(Observe that the mapping
\[
\varphi \mapsto \psi := \varphi \circ \Theta_F^{-1}: W_2^1(\Omega_0,\RRM) \rightarrow W_2^1(\Omega_F,\RRM)
\]
is an isomorphism.) Using that the mapping $F \mapsto (a_F, b_F)$ is smooth w.r.t. various topologies, it is straightforward to check that 
\[
F \mapsto u_{\gamma_F} \,|_{\gamma_F} \circ F = v_{\gamma_F}|_{\gamma_0} \circ F_0 \in C^{1-}(\mathbb{B}_{W_2^2(S,\RRM^2)}(0,\varepsilon), W_2^{1/2}(S,\RRM)), 
\]
cf. Lemma \ref{contws}, appendix. The Stefan problem reduces to the quasilinear parabolic evolution equation (for time dependent $\rho(t,\cdot)$, $F(t,\cdot)$) 
\[
\dot{\rho} = \frac{1}{(N_0 | N_{F(\rho)})} \left( - H_{F(\rho)} + u_{F(\rho)} \right), 
\]
which can be solved using standard techniques from parabolic theory. 

\vspace{5mm}

For later use we define 
\[
\mathbf{B}_\varepsilon := \mathbb{B}_{W_2^{2}(S,\RRM^2)}(F_0, \varepsilon), \qquad \mathbf{B}^*_\varepsilon := \mathbb{B}_{W_2^{2}(S,\RRM)}(0, \varepsilon)
\]
and keep in mind that (for suitable $\varepsilon_1, ..., \varepsilon_4$) 
\[
\rho \mapsto F \in C^\infty(\mathbf{B}^*_{\varepsilon_1}, \mathbf{B}_{\varepsilon_2}), \qquad \qquad F \mapsto \rho \in C^\infty(\mathbf{B}_{\varepsilon_3}, \mathbf{B}^*_{\varepsilon_4}).
\]

\vspace{5mm}

\section{An artificial continuous extension of the Stefan problem to the case of truly immersed curves.}

\vspace{5mm}

\subsection{Concept of a critical initial state.}

\vspace{5mm}

Let $S^+$ and $S^-$ be the closed upper and the lower half of the unit circle $S = \mathbb{R}/2\pi\mathbb{Z}$, respectively. Let $S^+ \subset X^+ \subset S$, $S^- \subset X^- \subset S$ be in such a way that $X^+ \cap X^-$ consists of two open disjoint components that are neigborhoods (in $S$) of $(1,0)$ and of $(-1,0)$. Let $F_0 = (x_0, y_0) \in C^\infty(S)$, $\gamma_0 := F_0[S] \subset \mathbb{R}^2$, $|F_0'| > 0$ on $S$. Assume: 
\begin{itemize}
\item $p = (p_1,p_2) \in \gamma_0$ iff $\check{p} = (p_1,-p_2) \in \gamma_0$.
\item $F_0[S^+] \subset \{ (x,y); \; y \geq 0 \}$, $F_0[S^-] \subset \{ (x,y); \; y \leq 0 \}$, $\gamma_0^+ := F[X^+]$, $\gamma_0^- := F[X^-]$ are embedded. 
\item There is precisely one line segment $K_0 \subset \{ (x,0); \; x > 0 \}$ s.t. $K_0 = ( \gamma_0^+ \cap \gamma_0^- ) \cap {\{ (x,y); \; x > 0 \}}$. We set $F_0[\mathcal{A}_0^+] := K_0 =: F_0[\mathcal{A}_0^-]$, where $\mathcal{A}_0^\pm \subset S^\pm$. $( \gamma_0^+ \cap \gamma_0^- ) \cap {\{ (x,y); \; x \leq 0 \}}$ consists of vertical lines.
\item There are connected open sets $S^\pm \supset \mathcal{O}_0^\pm \supset \mathcal{A}_0^\pm$ such that $x_0' < 0$ in $\mathcal{O}_0^-$ and $x_0' > 0$ in $\mathcal{O}_0^+$. The sets $\mathcal{O}_0^\pm \setminus \mathcal{A}_0^\pm$ are supposed to be small enough to obtain $\Vert y_0 \Vert_{C^1(\mathcal{O}_0^\pm,\RRM)} << 1$ and $N_{\gamma_0}|_{\mathcal{O}_0^\pm} \sim (0,\pm1)^T$, but are true supersets of $\mathcal{A}_0^\pm$, i.e. $y_0^\pm \left\{ \begin{smallmatrix} > \\ < \end{smallmatrix} \right\} 0$ on $\partial \mathcal{O}_0^\pm$.
\end{itemize}
\begin{minipage}{0.45\textwidth}
\begin{flushleft}
\includegraphics[width=5.0cm,height=6.25cm]{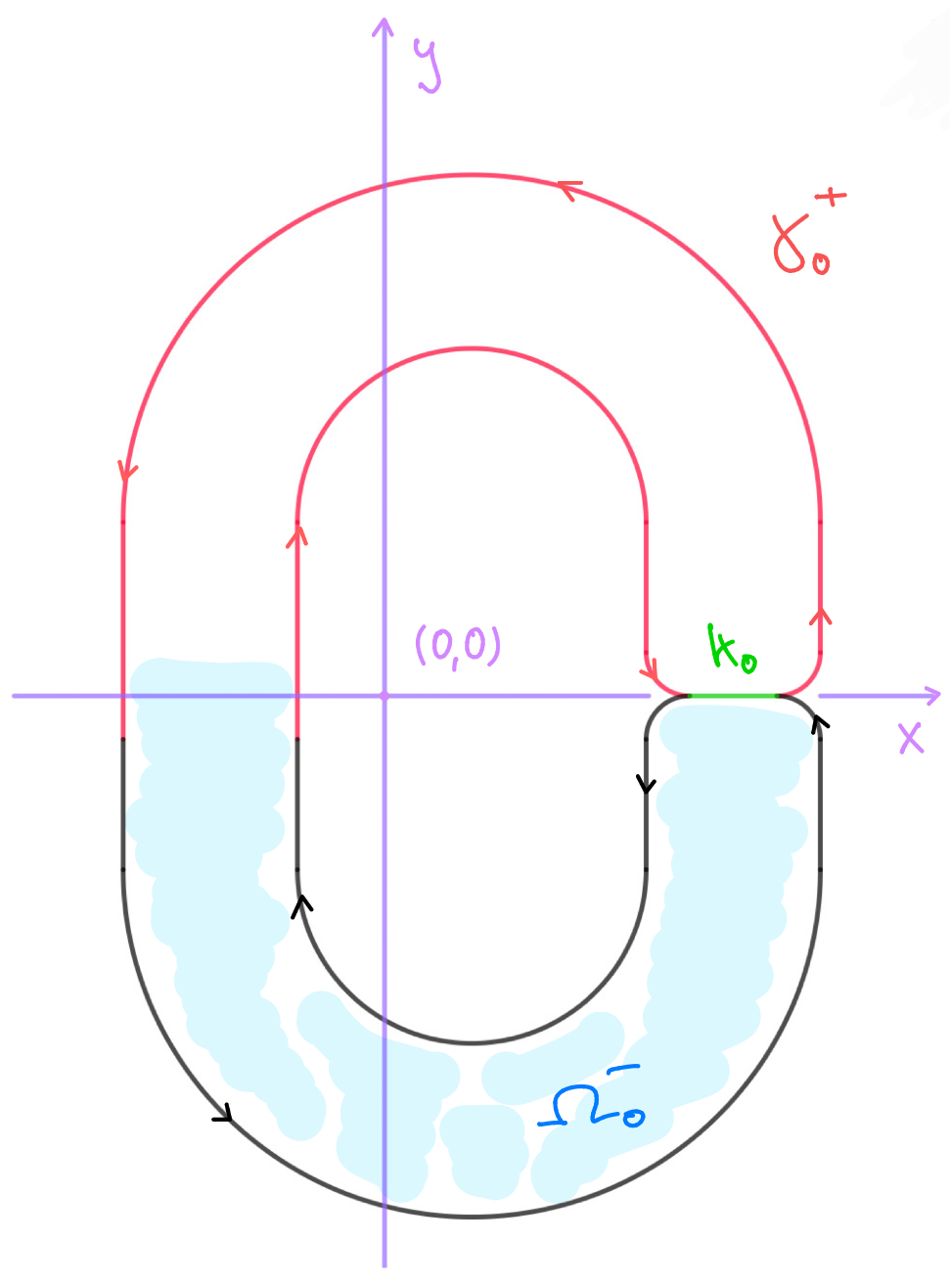}
\end{flushleft}
\end{minipage}
\begin{minipage}{0.05\textwidth}
$\;$
\end{minipage}
\begin{minipage}{0.5\textwidth}
\begin{flushleft}
\scriptsize{Concept of a possible initial shape. The realization of the condition \[ \min \left\{ \lim_{ \begin{smallmatrix} (x,y) \in\Omega_0^+ \\ (x,y)\downarrow E_0 \end{smallmatrix} } u_0(x,y), \; \lim_{ \begin{smallmatrix} (x,y) \in\Omega_0^- \\ (x,y)\uparrow E_0 \end{smallmatrix} } u_0(x,y) \right\} > 0\] for a subset $E_0 \subset K_0$ will be crucial, cf. section 4. ($u_0: \Omega_0 \rightarrow \RRM$ is temperature.)}
\end{flushleft}
\end{minipage}

\vspace{10mm}

For a curve $\gamma_0$ like this it makes sense to define the domain (an open connected subset of $\mathbb{R}^2$) $\Omega_0$ enclosed by $\gamma_0$. Also, the integral
\[
\int_{\Omega_0} \eta, \qquad \eta \in L_1(\Omega_0, \RRM), 
\]
is well defined. We split $\Omega_0 = \Omega_0^+ \cup \Omega_0^-$ in the following way: 
\begin{eqnarray*}
\Omega_0^\pm & := & \{ (x,y) \in \Omega_0; \; y \left\{ \begin{smallmatrix} > \\ < \end{smallmatrix} \right\} 0 \} \\
& \cup & \{ (x,y) \in \Omega_0; \; x < 0; \; \{ (z,y); \; z \in \RRM\} \cap \gamma_0^\pm \neq \emptyset \}. 
\end{eqnarray*}

\vspace{5mm}

\subsection{A trace mapping for curves with touching points.}

\vspace{5mm}

For a function $u_0 \in W_2^1(\Omega_0, \RRM)$ we can define the trace map as follows: Let $\{ \clu_1^\pm,...,\clu_m^\pm\}$ be open sets in the plane that cover $\gamma_0^\pm$ s.t. $\gamma_0^{j,\pm} := \gamma_0^\pm \cap \clu_j^\pm$ is embedded and connected, and s.t. the traces 
\[
\mbox{tr}_j^\pm \in \cll \left( W_2^1 \left( \clu_j^\pm \cap \Omega_0, \RRM \right), W_2^{1/2}\left( \gamma_0^{j,\pm}, \RRM  \right) \right)
\]
are well defined. (The latter can be checked by local flattening and using the half-space results $\mbox{tr} \in \cll ( W_2^1(\mathbb{H}^\pm,\RRM), W_2^{1/2}(\RRM,\RRM)  )$, where $\mathbb{H}^\pm := \{ (p_1,p_2) \in \RRM^2; \; p_2 \left\{ \begin{smallmatrix} > \\ < \end{smallmatrix} \right\} 0\}$.) Choosing suitable subordinate partitions of unity $\zeta_j^\pm$ this yields mappings
\[
\mbox{tr}^\pm: u \mapsto \left( \sum_{j=1}^m \zeta_j^\pm (\mbox{tr}_j^\pm u) \right) \circ F_0|_{X^\pm} \in \cll \left( W_2^1 \left( \Omega_0, \RRM \right), W_2^{1/2} \left( X^\pm , \RRM \right) \right).
\]
Since (by, again, a suitable choice of $\zeta_j^\pm$) $\mbox{tr}^+ u$ and $\mbox{tr}^- u$ coincide in $X^+ \cap X^-$, $\mbox{tr}^+$ and $\mbox{tr}^-$ can be put together to a mapping $\mbox{tr} \in \cll \left( W_2^1 \left( \Omega_0, \RRM \right), W_2^{1/2} \left( S , \RRM \right) \right)$ which we consider the trace-map w.r.t. $\gamma_0$. Note that the integral
\[
\int_{\gamma_0} u := \int_{\gamma_0} \mbox{tr} \, u := \int_S \mbox{tr} \, u(s) |F_0'| \; ds, \qquad \qquad u \in W_2^1(\Omega_0, \RRM), 
\]
is well defined. 
\begin{remark}
Note that the above definitions make also sense for a Lipschitz domain whose boundary curve contains touching points and is a.e. differentiable with essentially bounded first differential.
\end{remark}

\vspace{5mm}

{\subsection{Cutting of overlaps.}

\vspace{5mm} 

Let $F_0$ be as described above. As before, we consider $N_0$ both a function on $S$ and on $\gamma_0^\pm$, and indicate this by writing $N^\pm_{\gamma_0}(p)$ ($p = F_0(s)$) as well as $N_0^\pm(s)$, $s \in X^\pm$. Let  
$F := \left( \begin{smallmatrix} x \\ y \end{smallmatrix} \right) \in \mathbf{B}_\varepsilon$ be such that $\rho(F) \in \mathbf{B}_{\tilde \varepsilon}^*$ and, moreover, $\gamma^\pm_F := F[X^\pm]$ is a graph in normal direction over $\gamma^\pm_0$: 
\[
\rho^\pm(s) := \rho(s), \qquad \qquad \qquad\qquad \qquad \qquad \qquad \qquad \; \; \; s \in X^\pm, 
\]
\[
F^\pm(s) :=  \left( \begin{smallmatrix} x^\pm \\ y^\pm \end{smallmatrix} \right)(s) := F_0(s) - \rho^\pm(s) N^\pm_0(s), \qquad \qquad s \in X^\pm.   
\]
If $F \in \mathbf{B}_\varepsilon$, we have $(x^\pm)' \left\{ \begin{smallmatrix} > \\ < \end{smallmatrix} \right\} 0$, $x^\pm > 0$ on $\mathcal{O}_0^\pm$ and $y^\pm \left\{ \begin{smallmatrix} > \\ < \end{smallmatrix} \right\} 0$ on $\partial \mathcal{O}_0^\pm$. (Note that $\mathbf{B}_\varepsilon \hookrightarrow C^1(S,\RRM^2)$.) We define  
\[
\left\{ \begin{array}{lll} \mathcal{O}^+ = \mathcal{O}^+(F) & := & \{ s \in \mathcal{O}_0^+; \; y(s) < 0 \} \subset \mathcal{O}_0^+ \subset X^+, \\ \mathcal{O}^- = \mathcal{O}^-(F) & := & \{ s \in \mathcal{O}_0^-; \; y(s) > 0 \} \subset \mathcal{O}_0^- \subset X^-, \end{array} \right.
\]
\[
\tilde F^\pm (s) := \left\{ \begin{array}{ll} F(s), & s \in X^\pm \setminus \mathcal{O}^\pm \\ (x^\pm(s),0)^T & s \in \mathcal{O}^\pm, \end{array} \right.
\]
and $\tilde \gamma_F^\pm := \tilde F^\pm [X^\pm]$. For later use we also set 
\[
\tilde H_F (s) := \left\{ \begin{array}{ll} H_F(s), & s \in S \setminus (\mathcal{O}^+ \cup \mathcal{O}^-), \\ 0, & s \in \mathcal{O}^+ \cup \mathcal{O}^-. \end{array} \right.
\]
Let $N_0(s) := (n_0^1,n_0^2)^T(s)$. By definition, $\tilde y^\pm = 0$ in $\overline{\mathcal{O}^\pm}$ which gives rise to the ansatz  
\[
y_0(s) - \tilde \rho^\pm(s) n_0^2(s) = 0, \qquad \qquad \qquad s \in \overline{\mathcal{O}^\pm}, 
\]
and the definition 
\[
\tilde{\rho}^\pm(s) := \left\{ \begin{array}{ll} \rho^\pm(s), & s \in X^\pm \setminus \mathcal{O}^\pm, \\ \frac{y_0(s)}{n_0^2(s)}, & s \in \mathcal{O}^\pm. \end{array} \right. 
\]
(We again denote the functions
\[
p \mapsto \tilde \rho^\pm(s): \gamma_0 \rightarrow \RRM \qquad \qquad (p = F_0(s))
\]
by the same symbols $\tilde \rho^\pm$.) The curve $\tilde \gamma_F^\pm$ is parameterized by 
\[
\tilde F^\pm(s) = F_0(s) - \tilde{\rho}^\pm(s) N^\pm_0(s), \qquad \qquad s \in X^\pm,
\]
and $\tilde \gamma_F^\pm$ is a Lipschitz-graph over $\gamma_0^\pm$:  

\vspace{15mm}

\begin{minipage}{0.7\textwidth}
\begin{flushleft}
\includegraphics[width=10cm,height=5cm]{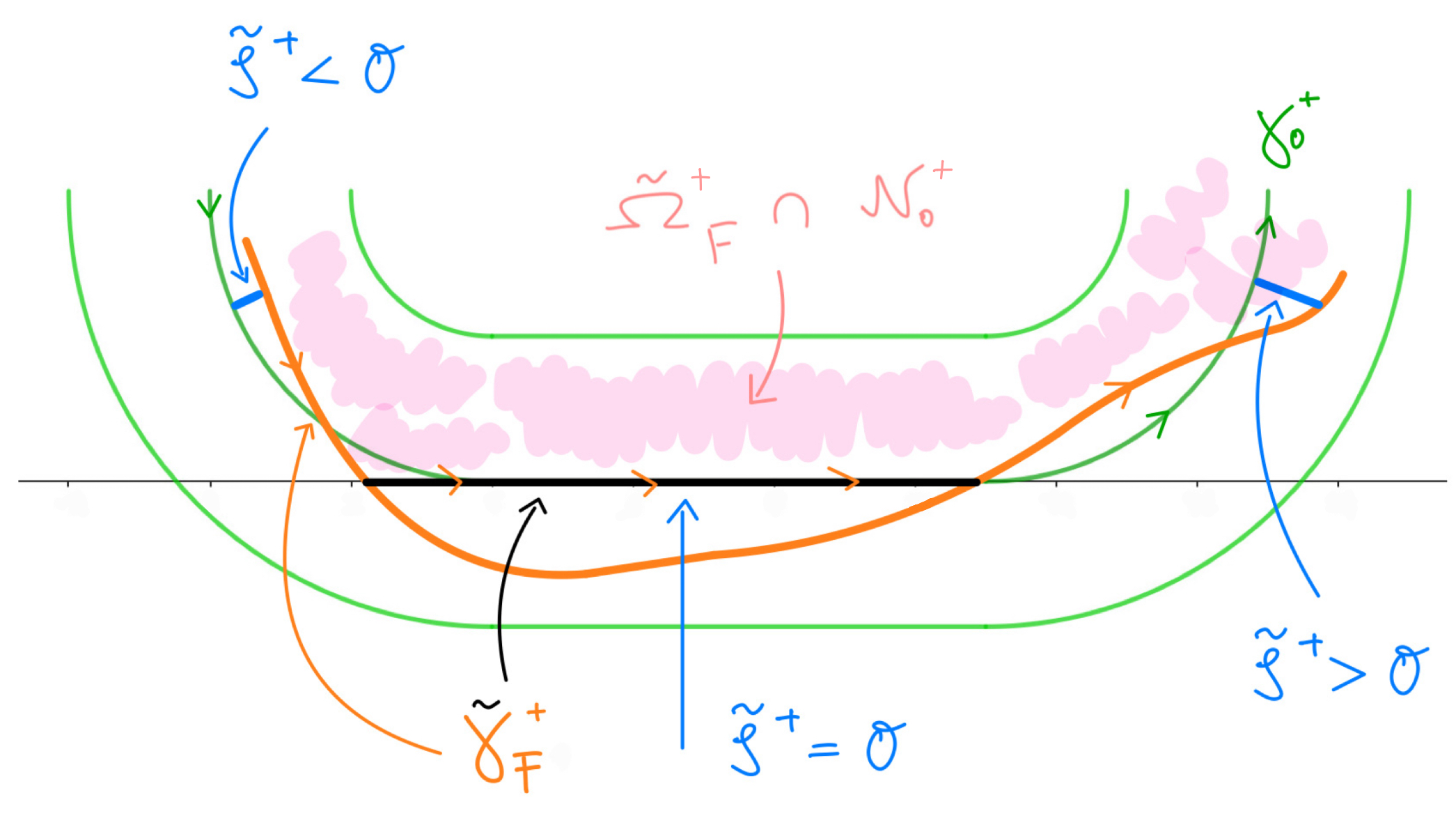}
\end{flushleft}
\end{minipage}
\begin{minipage}{0.1\textwidth}
$\;$
\end{minipage}
\begin{minipage}{0.2\textwidth}
\begin{flushleft}
\scriptsize{As defined later in the manuscript, $\mathcal{N}_0^+$ is a tubular neighborhood of $\gamma_0^+$.}
\end{flushleft}
\end{minipage}

\vspace{15mm}

The curve $\tilde \gamma_{F} := \tilde \gamma^+_{F} \cup \tilde \gamma^-_{F}$ will enclose an open connected domain $\tilde \Omega_F$ that we write as $\tilde \Omega_F = \tilde \Omega_F^+ \cup \tilde \Omega_F^-$, where
\begin{eqnarray*}
\tilde \Omega_F^\pm & := & \{ (x,y) \in \tilde \Omega_F; \; y \left\{ \begin{smallmatrix} > \\ < \end{smallmatrix} \right\} 0 \} \\
& \cup & \{ (x,y) \in \tilde \Omega_F; \; x < 0; \; \{ (z,y); \; z \in \RRM\} \cap \tilde \gamma_F^\pm \neq \emptyset \}. 
\end{eqnarray*}

\vspace{5mm}

\begin{lemma}\label{tsmall}
The functions $\tilde \rho^\pm$ are uniformly Lipschitz continuous w.r.t. $\rho \in \mathbf{B}_\varepsilon^\ast$ and $\rho \mapsto \tilde \rho^\pm \in B(\mathbf{B}_\varepsilon^\ast,W_\infty^1(X^\pm, \RRM))$. In fact, there is $\varepsilon_0 > 0$ s.t. $\Vert \tilde \rho^\pm \Vert_{W_\infty^1(X^\pm, \RRM)} \leq \mu < < 1$ for $\rho \in \mathbf{B}_\varepsilon^\ast$ and $\varepsilon \in (0, \varepsilon_0)$. 
\end{lemma}
\begin{proof}
We give the proof for $\tilde \rho^+$. Note that $\tilde \rho^+ = \rho^+$ on $X^+ \setminus \mathcal{O}_0^+$. Let $b_0 := y_0/n_0^2$. 

\vspace{5mm}

Choosing $\mathcal{O}_0^+ \setminus \mathcal{A}_0^+$ small enough, we have $n_0^2 |_{\mathcal{O}_0^+} \sim 1 > 0$ and $|b_0 |_{\mathcal{O}_0^+}|, |b_0' |_{\mathcal{O}_0^+}| << 1$. Observe that 
\[
(\tilde \rho^+ - b_0)|_{\mathcal{O}_0^+} := \min \{ (\rho^+ - b_0)|_{\mathcal{O}_0^+}, 0 \},  
\]
and therefore 
\[
|(\tilde \rho^+ - b_0)(s) - (\tilde \rho^+ - b_0)(t)| \leq |(\rho^+ - b_0)(s) - (\rho^+ - b_0)(t)|, \qquad \qquad s,t \in \mathcal{O}_0^+.
\]
Further, $\tilde \rho^+ = \rho^+$ is $C^1$ in $X^+ \setminus \mathcal{O}_0^+$, $\rho^+ = \tilde \rho^+$ on $\partial \mathcal{O}_0^+$ and 
\begin{eqnarray*}
|\tilde \rho^+(s) - \tilde \rho^+(t)| \leq |\rho^+(s) - \rho^+(r)| + |\tilde \rho^+(r) - \tilde \rho^+(t)|  
\end{eqnarray*}
$s \in X^+ \setminus \mathcal{O}_0^+$, $t \in \mathcal{O}_0^+$,  and $|r - s| = \mbox{dist}(s,\mathcal{O}_0^+)$. 

\vspace{5mm}

As a Lipschitz continuous function, $\tilde \rho^+$ is a.e. differentiable (Rademacher, \cite{Zie}). Thanks to $\mathbf{B}_\varepsilon^\ast \hookrightarrow C^1(S,\RRM)$, on $X^+ \setminus \mathcal{O}^+$ we have $|\tilde \rho^+|, \; |(\tilde \rho^+)'|  \sim C \varepsilon$ in this parameter range. The assertion follows from $\tilde \rho^+ = b_0$ in $\mathcal{O}^+$.
\end{proof}
\begin{definition}\label{defE}
$\mathbf{E}_\varepsilon^\ast$ is the closed subset of $\mathbf{B}_\varepsilon^\ast$ with elements s.t. $\rho^\pm = \tilde \rho^\pm$. $\hat{\mathbf{E}}_{\varepsilon}^\ast$ consists of those elements of $\rho \in \mathbf{E}_\varepsilon^\ast$ s.t. $F(\rho)[\mathcal{O}^\pm] \cap \{ (x,y) \in \RRM^2; \, y = 0 \} = \emptyset$. $\mathbf{E}_\varepsilon$ and $\hat{\mathbf{E}}_{\varepsilon}$ contain those  $F \in \mathbf{B}_{\varepsilon}$ s.t. $\tilde \rho(F) \in \mathbf{E}_{\varepsilon}^\ast$, $\tilde \rho(F) \in \hat{\mathbf{E}}_{\varepsilon}^\ast$, resp.
\end{definition}
Naturally, we obtain Lipschitz continuous homeomorphisms (uniform w.r.t $F \in \mathbf{B}_\varepsilon$) $\tilde \theta^\pm_{F}: \gamma_0^\pm \rightarrow \tilde \gamma^\pm_{F}$ by  
\[
\tilde \theta^\pm_{F}(p) := p - \tilde{\rho}^\pm(p) N^\pm_{\gamma_0}(p), \qquad \qquad p \in \gamma_0^\pm. 
\]
Recall that $\tilde F [S] = \tilde \gamma_F = \tilde \gamma_F^+ \cup \tilde \gamma_F^-$ and observe that the tangent spaces $T_p \tilde \gamma_F$ (defined by means of the parameterization $\tilde F$) exist for a.e. $p \in \tilde \gamma_F$.

\vspace{5mm}

Since $\tilde \rho^\pm = \rho^\pm$ on $F_0[X^+ \cap X^-]$, $\tilde \theta^\pm_{F}$ can be put together to a Lipschitz continuous homeomorphisms $\tilde \theta_{F}: \gamma_0 \rightarrow \tilde \gamma_F$ via 
\[
\tilde \theta_{F}(p) := \tilde \theta_{F}^\pm(p), \qquad \qquad p \in \gamma_0^\pm, 
\]
that satisfies
\begin{itemize}
\item[(i)] $\tilde \theta_{F_0} \circ F_0 = F_0$; 
\item[(ii)] $\tilde \theta_F \circ F_0 = \tilde F$; $(\tilde \theta_F \circ F_0)[S] = \tilde \gamma_F$; 
\item[(iii)] $\rho \mapsto \tilde \theta_{F(\rho)} \circ F_0 \in B(\mathbf{B}_\varepsilon^\ast, \mbox{Lip}(S, \RRM^2)) \cap C(\mathbf{E}_\varepsilon^\ast, \mbox{Lip}(S, \RRM^2))$; 
\item[(iv)] $\tilde \theta_F \circ F_0$ is differentiable a.e. on $S$ and $\rho \mapsto D (\tilde \theta_{F(\rho)} \circ F_0) \in B(\mathbf{B}_\varepsilon^\ast, L_\infty(S, \RRM^2)) \cap C(\mathbf{E}_\varepsilon^\ast, L_\infty(S, \RRM^2))$;
\item[(v)] For $\rho \in \mathbf{B}_\varepsilon^\ast$,  
\begin{eqnarray}\label{stheta}
\norm{ (\tilde \theta_{F(\rho)} - \mathbf{Id}) \circ F_0 }_{C(S, \RRM^2)} + \norm{ D (\tilde \theta_{F(\rho)} \circ F_0) - F_0' }_{L_\infty(S, \RRM^2)} \nonumber \\ \nonumber \\ << 1; 
\end{eqnarray}
\item[(vi)] If $F \in \hat{\mathbf{E}}_\varepsilon$ and $u \in L_1(\tilde \gamma_F,\RRM)$ we have 
\[
\int_{\tilde \gamma_F} u = \int_S |D (\tilde \theta_F \circ F_0)| u(\tilde \theta_F \circ F_0).
\]
\end{itemize}
\begin{remark}
If $X, Y$ are Banach spaces, we denote by $B(X,Y)$ the set of bounded and Borel - measurabel mappings $X \rightarrow Y$. Note that the above mappings (iii), (iv) are not continuous w.r.t. the adressed topologies on the whole ball $\mathbf{B}_\varepsilon^\ast$. The measurability basically is a consequence of the fact that $\partial (\tilde \rho^\pm - b_0)|_{\mathcal{O}_0} = (\rho - b_0)'|_{\mathcal{O}_0} \cdot \mathbf{1}_{ \left\{ (\rho - b_0)|_{\mathcal{O}_0} \left\{ \begin{smallmatrix} < \\ > \end{smallmatrix} \right\} 0 \right\}}$ a.e. . We also have 
\begin{eqnarray}\label{regH}
\rho \mapsto \tilde H_{F(\rho)} \in B(\mathbf{B}_\varepsilon^\ast, L_2(S)) \cap C(\mathbf{E}_\varepsilon^\ast, L_2(S)).
\end{eqnarray}
\end{remark}

\vspace{10mm}

We denote by $\mathcal{N}_0^\pm$ tubular neighborhoods (with radius $a_0$) of $\gamma_0^\pm$, by $P_0^\pm, \Lambda_0^\pm$ the corresponding projections and signed distance functions, and by $\left( \Phi_0^\pm \right)^{-1} := (P_0^\pm, \Lambda_0^\pm)$ the respective inverse of the diffeomorphism $\gamma_0^\pm \times (-a_0,a_0) \rightarrow \mathcal{N}_0^\pm$, 
\[
\Phi_0^\pm(p,\delta) := p - \delta N^\pm_{\gamma_0}(p).
\]
We construct Lipschitz continuous homeomorphisms $\mathcal{N}_0^\pm \cap \Omega_0^\pm \rightarrow \mathcal{N}_0^\pm \cap \tilde \Omega_F^\pm$ as follows: We choose as cut-off function $\varphi_0 \in C^\infty([-a_0,0],[0,1])$, $\varphi_0(r) := 1 + r/a_0$. (Then $|\varphi_0'| =  1/a_0$.) Analogue to the embedded situation we can choose the mapping 
\[
\tilde \Theta_F^\pm(p) := P_0^\pm(p) - \left[ \Lambda^\pm_0(p) + \varphi_0 \left( \Lambda_0(p) \right) \tilde \rho^\pm \left( P^\pm_0(p) \right) \right] N_{\gamma_0}^\pm(P^\pm_0(p)), 
\]
$p \in \mathcal{N}_0^\pm$, and obtain 
\[
\tilde \Theta_F^\pm |_{\gamma_0^\pm} = \tilde \theta_F^\pm, \qquad \qquad \tilde \Theta_F^\pm |_{\gamma_0^\pm + a_0 N^\pm_{\gamma_0}} = \mathbf{Id}. 
\]
\begin{remark}
Observe that if $F \in \mathbf{B}_\varepsilon$ (and $\varepsilon > 0$ is small enough), then $|\tilde \rho^\pm| < < a_0/2$.
\end{remark}
\begin{lemma}
The homeomorphisms $\tilde \Theta_F^\pm$ are  uniformly Bi-Lipschitz w.r.t. $\rho \in \mathbf{B}_\varepsilon^\ast$.
\end{lemma}
\begin{proof}
Let $\Psi_{\tilde \rho^\pm} := \Psi: \gamma_0^\pm \times [- a_0,0] \rightarrow \bigcup_{p \in \gamma_0^\pm} \{ p \} \times [-a_0,\tilde \rho^\pm(p)]$ be defined by 
\[
\Psi(p,r) := (p,\psi_p(r)) := (p, r + \varphi_0(r)\tilde \rho^\pm(p)). 
\]
Since 
\begin{equation}\label{decomp}
\tilde \Theta_F^\pm|_{\mathcal{N}_0^\pm \cap \Omega_0^\pm} = \Phi_0^\pm \circ \Psi \circ \left( \Phi_0^\pm \right)^{-1}, 
\end{equation}
it suffices to show that $\Psi$ is uniformly Bi-Lipschitz. (In this proof we write $\tilde \rho$ instead of $\tilde \rho^\pm$ for simplicity.) Since $\varphi_0$ is smooth on $[-a_0,0]$ and by Lemma \ref{tsmall},
\[
|\tilde \rho(p) - \tilde \rho(q)| \leq C d_{\gamma_0}(p,q), \qquad \qquad \rho \in \mathbf{B}_\varepsilon^\ast, 
\]
$\Psi$ is uniformly Lipschitz, and it suffices to show the same for $\Psi^{-1}$. Observe that 
\[
\Psi^{-1}(q,s) = (q,\psi_q^{-1}(s)) = \left( q,\frac{s - \tilde \rho(q)}{1 + \tilde \rho(q) / a_0} \right)
\]
and , since $|\tilde \rho| << a_0 / 2$, 
\begin{eqnarray*}
|\psi_p^{-1}(r) - \psi_q^{-1}(s)| & = & \frac{|(r - \tilde \rho(p))(1 + \tilde \rho(q) / a_0) - (s - \tilde \rho(q))(1 + \tilde \rho(p) / a_0)|}{|1 + \tilde \rho(p) / a_0| |1 + \tilde \rho(q) / a_0 |} \\ 
& \leq & 4 ( 2 |r-s| + (1 + s / a_0)|\tilde \rho(p) - \tilde \rho(q)| ) \\
& \leq & 8 ( |r-s| + (|\tilde \rho(p) - \tilde \rho(q)| ).
\end{eqnarray*}
This implies the assertion. 
\end{proof}
By Rademachers Theorem, the homeomorphisms $\tilde \Theta_F^\pm: \mathcal{N}_0^\pm \cap \Omega_0^\pm \rightarrow \mathcal{N}_0^\pm \cap \tilde \Omega_F^\pm$ (and their respective inverses) are almost everywhere differentiable, and the structure of the first differential of $\Psi$ and $\Psi^{-1}$,  
\[
d_{(p,r)} \Psi = (v,\xi) \mapsto \begin{pmatrix} v \\ \varphi_0(r) (d_p \tilde \rho^\pm)(v) + \xi (1 + \varphi_0'(r) \tilde \rho^\pm(p)) \end{pmatrix} 
\]
($(p,r) \in \gamma_0 \times \RRM$, $(v,\xi) \in T_p \gamma_0 \times \RRM$),
\[
d_{(q,s)} \Psi^{-1} = (u,\zeta) \mapsto \begin{pmatrix} u \\ (d_q \psi_q^{-1}(s))(u) + \zeta ( 1 + a_0 / ( a_0 + \tilde \rho^\pm(q)) ) \end{pmatrix} 
\]
($(q,s) \in \tilde \gamma_F \times \RRM$, $(u,\zeta) \overset{ \mbox{\tiny{a.e.}} }{\in} T_q \tilde \gamma_F \times \RRM$),
implies that 
\begin{equation}\label{D1}
\norm{ D \tilde \Theta_F^\pm }_{L_\infty(\mathcal{N}_0^\pm \cap \Omega_0^\pm, \mathcal{L}(\mathbb{R}^2))} +  \norm{ D \left( \tilde \Theta_F^\pm \right)^{-1} }_{L_\infty(\mathcal{N}_0^\pm \cap \tilde \Omega^\pm_F, \mathcal{L}(\mathbb{R}^2))} \leq L,  
\end{equation}
\begin{equation}\label{D2}
\norm{ \tilde \Theta_F^\pm }_{\mbox{\scriptsize{Lip}}(\mathcal{N}_0^\pm \cap \Omega^\pm_0, \mathcal{N}_0^\pm \cap \tilde{\Omega}_F^\pm)} +  \norm{ \left( \tilde \Theta_F^\pm \right)^{-1} }_{\mbox{\scriptsize{Lip}}(\mathcal{N}_0^\pm \cap \tilde{\Omega}_F^\pm, \mathcal{N}_0^\pm \cap \Omega_0^\pm)} \leq L,   
\end{equation}
$F \in \mathbf{B}_\varepsilon$. Moreover, in view of Lemma \ref{tsmall}, 
\begin{eqnarray}\label{D3}
\norm{ \tilde \Theta_F^\pm - \mathbf{Id} }_{BUC(\mathcal{N}_0^\pm \cap \Omega^\pm_0, \mathcal{N}_0^\pm \cap \tilde{\Omega}_F^\pm)} + \norm{ D \tilde \Theta_F^\pm - \mathbf{Id} }_{L_\infty(\mathcal{N}_0^\pm \cap \Omega_0^\pm, \mathcal{L}(\mathbb{R}^2))} \nonumber \\ \nonumber  \\ << 1/L, 
\end{eqnarray}
\begin{eqnarray}\label{D4}
\norm{ \left(\tilde \Theta_F^\pm \right)^{-1} - \mathbf{Id} }_{BUC(\mathcal{N}_0^\pm \cap \tilde{\Omega}_F^\pm, \mathcal{N}_0^\pm \cap \Omega^\pm_0)} + \norm{ D \left( \tilde \Theta_F^\pm \right)^{-1} - \mathbf{Id} }_{L_\infty(\mathcal{N}_0^\pm \cap \tilde \Omega_F^\pm, \mathcal{L}(\mathbb{R}^2))} \nonumber \\ \nonumber  \\ << 1/L,  
\end{eqnarray}
$F \in \mathbf{B}_\varepsilon$. Gluing the homeomorphisms together via 
\[
\tilde \Theta_{F}(p) := \tilde \Theta_{F}^\pm(p), \qquad \qquad p \in \mathcal{N}_0^\pm \cap \Omega_0^\pm 
\]
( which makes sense, since 
\[
\tilde \Theta_{F}^+(p) = \tilde \Theta_{F}^-(p), \qquad \; p \in \{ F_0(s) + a \cdot N_0(s); \; s\in X^+ \cap X^-; \; a \in [0,a_0] \} \mbox{  ),}
\]
extension to $\Omega_0$ by the identity yields a Lipschitz continuous homeomorphism
\[
\tilde \Theta_F: \Omega_0 \rightarrow \tilde \Omega_F 
\]  
enjoying the analogue estimates. Moreover, 
\begin{eqnarray}\label{C1}
\rho \mapsto \tilde \Theta_{F(\rho)} & \in & B \left( \mathbf{B}^\ast_\varepsilon, \mbox{\small{Lip}}(\Omega_0, \tilde \Omega_{F}) \right) \cap C \left( \mathbf{E}^\ast_\varepsilon, \mbox{\small{Lip}}(\Omega_0, \tilde \Omega_{F}) \right),
\end{eqnarray}
\begin{eqnarray}\label{C2}
\rho \mapsto \left( D \tilde \Theta_{F(\rho)}, D \tilde \Theta_{F(\rho)}^{-1} \circ \tilde \Theta_{F(\rho)}\right) & \in & B \left( \mathbf{B}^\ast_\varepsilon, \left[ L_\infty(\Omega_0, \mathcal{L}(\mathbb{R}^2)) \right]^2 \right) \nonumber \\
& \cap & C \left( \mathbf{E}^\ast_\varepsilon, \left[ L_\infty(\Omega_0, \mathcal{L}(\mathbb{R}^2)) \right]^2 \right).
\end{eqnarray}

\begin{remark}
From now on we use the short notation $W_p^s(X) := W_p^s(X,\RRM)$ where $X \in \{ S, \Omega_0, \tilde \Omega_F \}$.
\end{remark}

We are now in the position to proceed analogue to the embedded situation from the beginning: Let $F \in \mathbf{B}_\varepsilon$. Then temperature in the extended Stefan problem is defined by the operator equation 
\begin{equation}\label{eSpB}
\tilde a_F(v,\varphi) = \tilde b_F(\varphi), \qquad \qquad \varphi \in W_2^1(\Omega_0), 
\end{equation}
where 
\begin{eqnarray*}
\tilde a_F(v,\varphi) & := & \int_{\Omega_0} |\mbox{det} D \tilde \Theta_F | ( \nabla v | \nabla \varphi )_{\tilde F} + \int_{S}|D (\tilde \theta_F \circ F_0)| (v \varphi)(F_0), \\
\tilde b_F(\varphi) & := & \int_{S} |D (\tilde \theta_F \circ F_0)| \tilde H_F \varphi(F_0) \\ 
\end{eqnarray*}
and  
\[
\left( \nabla v | \nabla \varphi \right)_{\tilde F} = \left( \; \left( D \tilde \Theta_F^{-1}\right)^T (\tilde \Theta_F) \nabla v \; | \;  \left( D \tilde \Theta_F^{-1} \right)^T (\tilde \Theta_F) \nabla \varphi \; \right).
\]
By (\ref{D1}) - (\ref{D4}) and (\ref{stheta}), the form $\tilde a_F: W_2^1(\Omega_0) \times W_2^1(\Omega_0) \rightarrow \mathbb{R}$ is coercive and bounded, and $\tilde b_F: W_2^1(\Omega_0) \rightarrow \mathbb{R}$  is bounded uniformly w.r.t. $F \in \mathbf{B}_\varepsilon$: The uniform boundedness of $\tilde a_F$ and $\tilde b_F$ is trivial, the uniform coercivity of $\tilde a_F$ is shown in the appendix.

\vspace{5mm}

We name $\tilde v_{\gamma_F} \in W_2^1(\Omega_0)$ the unique solution of (\ref{eSpB}). 
\begin{remark}\label{regun}
In case that $F$ is an embedding (i.e. $\tilde \rho(F) = \rho(F) \in \hat{\mathbf{E}}_\varepsilon^\ast$), $u_{\gamma_F} := \tilde v_{\gamma_F} \circ \tilde \Theta_{F}^{-1} \in W_2^1(\Omega_{F})$ is the unique weak solution of  
\[
\Delta u  = 0 \mbox{ in } \Omega_F, \qquad \qquad u - \partial_{N_{\gamma_F}} u = H_{\gamma_F} \mbox{ on } \gamma_F, 
\]
where $H_{\gamma_F} \in L_2(\gamma_F, \RRM)$, and in case $\gamma_F$ is smooth, $u_{\gamma_F} \in C^\infty(\overline{\Omega}_F,\RRM)$. Moreover, 
\[
u_F := u_{\gamma_F} |_{\gamma_F} \circ F = \tilde v_{\gamma_F} \circ \tilde \Theta_{F}^{-1} |_{\gamma_F} \circ (\tilde \theta_F \circ F_0) = \tilde v_{\gamma_F} |_{\gamma_0} \circ F_0.
\]
\end{remark}

\vspace{5mm}

The term $\tilde v_{\gamma_F} |_{\gamma_0} \circ F_0$ makes also sense in the non embedded situation (cf. section 2.1), and it is straightforward to check (using (\ref{D1}) - (\ref{C2}), (\ref{regH}) and  Lemma \ref{contws}, appendix) that
\begin{equation}\label{C3}
\rho \mapsto \tilde u_{F(\rho)} := \tilde v_{\gamma_{F(\rho)}} \,|_{\gamma_0} \circ F_0 \in B(\mathbf{B}_\varepsilon^*, W_2^{1/2}(S)) \cap C(\mathbf{E}_\varepsilon^*, W_2^{1/2}(S)).
\end{equation}
The extended Stefan problem becomes the quasilinear evolution equation 
\[
\dot{\rho} = \frac{1}{(N_0 | N_{F(\rho)})} \left( - H_{F(\rho)} + \tilde u_{F(\rho)} \right). 
\]
It is important for the solution theory, that $N_{F(\rho)}$ and $H_{F(\rho)}$ do not have a tilde.  

\vspace{5mm}

\section{The fixed point argument}

\vspace{5mm}

The principle part of $H_{F(\rho)}$ is 
\[
\frac{- \rho''}{|F_0'|^4} \left( y_0' \left( y_0 - \frac{\rho x_0'}{|F_0'|} \right)' + x_0' \left( x_0 + \frac{\rho y_0'}{|F_0'|} \right)' \right).
\] 
Oberve that 
\[
\sigma \mapsto \lambda(\sigma) := \frac{1}{(N_0 | N_{F(\sigma)})} \in C^\infty(\mathbf{B}^*_\varepsilon,W_2^1(S))
\]
and
\[
|\lambda(\sigma) - 1 | << 1 \qquad \qquad (\sigma \in \mathbf{B}^*_\varepsilon).
\]
Let $\sigma \in \mathbf{B}^*_\varepsilon$ (if necessary identified with the constant mapping $t \mapsto \sigma$ ($t \in \mathbb{R}$). For suitable $\omega \in \RRM$ (to be determined later)  
define the operator $P_\omega := P \in C^\infty( \mathbf{B}_\varepsilon^\ast, \mathcal{L}(W_2^{2+\delta}(S),W^{\delta}_2(S)) )$ ($\delta \in [0,1/4]$) by 
\[
P_\omega(\sigma) := P(\sigma) := \rho \mapsto \frac{\lambda(\sigma)}{|F_0'|^4} \left( y_0' \left( y_0 - \frac{\sigma x_0'}{|F_0'|} \right)' + x_0' \left( x_0 + \frac{\sigma y_0'}{|F_0'|} \right)' \right) \rho'' - \omega \rho. 
\]
Choose 
\[
Q_\omega := Q \in C^\infty(W_2^2(S), W_2^1(S))
\]
by 
\[
- \lambda(\sigma) H_{F(\sigma)} = P(\sigma)\sigma + Q(\sigma).
\]
(Observe that $W_2^\alpha(S)$ is an algebra w.r.t. pointwise multiplication provided $\alpha > 1/2$.) We can rewrite the extended Stefan problem as
\begin{equation}
\dot{\rho} - P(\rho)\rho = Q(\rho) + \lambda(\rho) \tilde u_{F(\rho)}, \qquad \qquad \rho(0) = \rho_0.
\end{equation}

\vspace{5mm}

\subsection{Time trace zero}

\vspace{5mm}

For this section we introduce the balls $\undertilde{\mathbf{B}}^*_\varepsilon := \mathbb{B}_{W_2^{2+1/4}(S,\RRM)}(0, \varepsilon)$. The sets $\undertilde{\mathbf{E}}^*_\varepsilon$, $\undertilde{ \hat{\mathbf{E}}}^*_\varepsilon$ are defined in the obvious way.

\vspace{5mm}

Let $\zeta := \rho - \rho_0$. We aim to solve the equation 
\begin{equation}
\dot{\zeta} - P(\zeta + \rho_0)\zeta = P(\zeta + \rho_0)\rho_0 + Q(\zeta + \rho_0) + \tilde w(\zeta+ \rho_0)
\end{equation}
with the initial demand $\zeta(0) = 0$ and $\tilde w(\zeta+ \rho_0) := \lambda(\zeta + \rho_0) \tilde u_{F(\zeta + \rho_0)}$.
Observe that the spaces $W_2^{2+1/8}$, $W_2^2(S)$, $W_2^1(S)$, $W_2^{1/2}(S)$ are interpolation spaces of exponents $15/16$, $7/8$, $3/8$, $1/8$, respectively, w.r.t. the interpolation couple $\left( W_2^{1/4}(S), W_2^{2+1/4}(S) \right)$. Because of the embedding $W_2^1(S) \hookrightarrow C^{3/8}(S) \hookrightarrow W_2^{1/4}(S)$, elements of $W_2^1(S)$ are multipliers for $W_2^{1/4}(S)$, i.e. 
\[
W_2^1(S) \cdot W_2^{1/4}(S) \hookrightarrow W_2^{1/4}(S).
\]
From this one concludes by standard arguments, 
\[
-P(\zeta + \rho_0) \subset \mathcal{H}\left(W_2^{2+1/4}(S),W_2^{1/4}(S)\right), \qquad \qquad \zeta + \rho_0: [0,\tau] \rightarrow \mathbf{B}_\varepsilon^\ast, 
\]
where $\mathcal{H}(E,F)$ denotes the set of negative generators $A \in \mathcal{L}(E,F)$ of analytic semigroups.

\vspace{5mm}

Consider for a fixed $\delta_0 > 0$ the bounded, closed, convex subset of $C([0,\tau],W_2^2(S))$ defined by 
\begin{eqnarray*}
& & \mathcal{M}(\delta_0, \tau) := \mathcal{M}(\tau)  := \\
& & \{ \zeta \in C([0,\tau],W_2^2(S)), \quad \zeta(0) = 0, \quad \Vert \zeta(t) \Vert_{W_2^2(S)} \leq \delta_0, \\
& & \Vert \zeta(t)-\zeta(s) \Vert_{W_2^2(S)} \leq |t-s|^{1/32}, \quad t,s \in [0,\tau] \}.
\end{eqnarray*}
According to the results from sections II 4, 5 in \cite{LaQPP} (namely II 4.4.1, II 4.4.2, II 5.1.3), for each $\zeta \in \mathcal{M}(\tau)$ there is a parabolic evolution operator $U_{P(\zeta + \rho_0)}$ associated with $P(\zeta + \rho_0)$ that satisfies the following uniform estimates (here we have chosen $\omega \in \RRM$ suitably in order to pin the resolvent sets appropriately):
\begin{eqnarray*}
    & \Vert U_{P(\zeta + \rho_0)}(t,s) \Vert_{\mathcal{L}\left(W_2^{1/4}(S), W_2^2(S)\right)} \\
 + & \Vert U_{P(\zeta + \rho_0)}(t,s) \Vert_{\mathcal{L}\left(W_2^{1/2}(S), W_2^2(S)\right)}\\
 + & \Vert U_{P(\zeta + \rho_0)}(t,s) \Vert_{\mathcal{L}\left(W_2^{1}(S), W_2^2(S)\right)}  \\
\leq & M \left( |t-s|^{-7/8} + |t-s|^{-3/4} + |t-s|^{-1/2} \right), 
\end{eqnarray*}
\begin{eqnarray*}
    & \Vert U_{P(\zeta + \rho_0)}(t,s) \Vert_{\mathcal{L}\left(W_2^{1/4}(S), W_2^{2+1/8}(S)\right)} \\
 + & \Vert U_{P(\zeta + \rho_0)}(t,s) \Vert_{\mathcal{L}\left(W_2^{1/2}(S), W_2^{2+1/8}(S)\right)} \\
 + & \Vert U_{P(\zeta + \rho_0)}(t,s) \Vert_{\mathcal{L}\left(W^{1}_{2}(S), W_2^{2+1/8}(S)\right)} \\
\leq & M \left( |t-s|^{-15/16} + |t-s|^{-13/16} + |t-s|^{-9/16} \right)
\end{eqnarray*}
and 
\begin{equation}\label{equalreg}
\Vert U_{P(\zeta + \rho_0)}(t,s) \Vert_{\mathcal{L}\left(W^{\alpha}_{2}(S), W_2^{\alpha}(S)\right)} \leq M, \qquad \qquad \alpha \in [1/4,2+1/4], 
\end{equation}
where $M$ is independent of $\zeta \in \mathcal{M}(\tau)$, $\rho_0 \in \undertilde{\mathbf{B}}^*_\varepsilon$ and $\tau \in [0,T]$. Solving our problem is equivalent to identifying a fixed point of the mapping $\zeta \mapsto \xi$ where $\xi$ solves 
\begin{equation}\label{xi}
\xi_t - P(\zeta + \rho_0)\xi = P(\zeta + \rho_0)\rho_0 + Q(\zeta + \rho_0) + \tilde w(\zeta + \rho_0)  
\end{equation}
and by variation of constants is given (as a mild solution) by
\begin{eqnarray}\label{xi2}
\\ \xi(t) = \int_0^t U_{P(\zeta + \rho_0)}(t,\tau) \left( P(\zeta + \rho_0)\rho_0 + Q(\zeta + \rho_0) + \tilde w(\zeta + \rho_0) \right)(\tau,\cdot) \; d\tau. & \nonumber   
\end{eqnarray}
Note that we may assume that 
\[
\Vert Q(\zeta + \rho_0)(t,\cdot) \Vert_{W_2^1(S)} + \Vert \tilde w(\zeta + \rho_0)(t,\cdot) \Vert_{W_2^{1/2}(S)} \leq M,  
\]
\[
\Vert P(\zeta + \rho_0)\rho_0 \Vert_{W_2^{1/4}(S)} \leq M, \qquad \qquad (\zeta,t,\rho_0) \in \mathcal{M}(\tau) \times [0,\tau] \times \undertilde{\mathbf{B}}^*_\varepsilon.
\]
Hence, 
\[
\Vert \xi(t) \Vert_{W_2^{2}(S)} \leq 2 M^2 \left( t^{1/8} + t^{1/2} + t^{1/4} \right), \qquad \qquad t \in [0,\tau], 
\]
and
\[
\Vert \xi(t) \Vert_{W_2^{2+1/8}(S)} \leq 2 M^2 \left( t^{1/16} + t^{7/16} + t^{3/16} \right), \qquad \qquad t \in [0,\tau].
\]
(The upper bounds are both $\leq \delta_0$, if $\tau > 0$ is small enough.) 
Theorem II 5.3.1 in \cite{LaQPP} yields 
\begin{eqnarray*}
\Vert \xi(t) - \xi(s) \Vert_{W_2^2(S)} & \leq & C |t-s|^{1/16} \\
& \leq & C (2\tau)^{1/32} |t-s|^{1/32}, 
\end{eqnarray*}
where $C > 0$ is independent of $(\zeta,t,\rho_0) \in \mathcal{M}(\tau) \times [0,\tau] \times \undertilde{\mathbf{B}}^*_\varepsilon$. Therefore, if $\tau > 0$ is small enough, then
\[
\zeta \mapsto \xi: \mathcal{M}(\tau) \rightarrow \mathcal{M}(\tau).
\]
By the compactness of the embedding $W_2^{2+1/8}(S) \hookrightarrow W_2^2(S)$ and the Arzel\`{a} - Ascoli theorem, we obtain the existence of the desired fixed point $\xi \in \mathcal{M}(\tau)$ by the Schauder fixed point Theorem.  (In the sequel we assume that $\delta_0 > 0$ in the definition of $\mathcal{M}(\tau) = \mathcal{M}(\delta_0, \tau)$ is sufficiently small.)

\vspace{5mm}

Consider now a sequence $\rho_0^n \rightarrow 0$ in $\undertilde{\hat{\mathbf{E}}}_\varepsilon^\ast \subset \undertilde{\mathbf{B}}^*_\varepsilon$. (The $\rho_0^n$ correspond to embedded initial states.) The corresponding solutions $\rho_n := \rho_0^n + \xi_n$ lie (by almost the same argument as above) relatively compact in $\mathcal{M}(\tau)$. Therefore, after passing to a subsequence again denoted by $\rho_n$, $\rho_n \rightarrow \rho_\infty \in \mathcal{M}(\tau)$. Assume that 
\begin{equation}\label{Ann}
\bigcup_{n > 0, \; t \in [0,\tau]} \rho_n(t) \subset \mathbf{E}_\varepsilon^\ast.
\tag{A}
\end{equation}
Then $\rho_\infty \in C([0,\tau],\mathbf{E}_\varepsilon^\ast)$. Since $\rho \mapsto \tilde u_{F(\rho)} \in C(\mathbf{E}_\varepsilon^\ast,W_2^{1/2}(S))$, it is a direct consequence of (\ref{xi2}) and Lemma II 5.1.4 in \cite{LaQPP} that 
\[
\rho_\infty(t) = \int_0^t U_{P(\rho_\infty)}(t,\tau) \left( Q(\rho_\infty) + \tilde w(\rho_\infty) \right)(\tau,\cdot) \; d\tau, 
\]
$\rho_\infty \in C^1((0,\tau],L_2(S))$ (Theorem II 1.2.2 in \cite{LaQPP}) and $\rho_\infty(0) = 0$. Since $\mathcal{M}(\tau) \hookrightarrow C([0,\tau] \times S)$, given $\eta > 0$ there is $n \in \mathbb{N}$ s.t. 
\begin{equation}\label{sing}
\sup_{(t,s) \in [0,\tau] \times S} |\rho_\infty(t,s) - \rho_n(t,s)| < \eta. 
\end{equation}
Moreover, 
\[
\dot{\rho}_\infty = \lambda(\rho_\infty) \left( - H_{F(\rho_\infty)} + \tilde u_{F(\rho_\infty)} \right) \in C([0,\tau],L_2(S)),   
\]
but $\dot{\rho}_{\infty}(0) = - H_{F_0} + \tilde u_{F_0}$ is actually smooth. Suppose that, in addition, $\dot{\rho}_{\infty}(0)$ is positive on some small open sets $U^\pm \subset \mathcal{A}_0^\pm$. This implies $\rho_\infty(\tau_0,s_0^\pm) := \eta_0 > 0$ for some $(\tau_0,s_0^\pm) \in (0,\tau] \times U^\pm$: 

\vspace{5mm}

Indeed, by the regularity of $\dot{\rho}_\infty$, 
\[
\frac{1}{t_n} (\rho_\infty(t_n,s) - \rho_\infty(0,s)) \rightarrow \dot{\rho}_\infty(0,s) \qquad \qquad  \mbox{ a.e. on } U^\pm.
\]
By (\ref{sing}), 
\[
\rho_n(\tau_0,s_0^\pm) > \eta_0/2 \qquad \qquad \mbox{ for $n > 0$ large enough. } 
\]
This contradicts (A) and verifies $\lnot$(A):
\[
\exists \; (n,\tau_0) \in \mathbb{N} \times [0,\tau]: \quad \rho_n(\tau_0) \neq \tilde \rho_n(\tau_0).
\]
Alltogether:

\vspace{5mm}
 
\begin{theorem}
There exists a time $T > 0$ with the following properties: 
\begin{itemize}
\item[1)] Let $\rho_0 := 0$. If $u_0 \in W_2^1(\Omega_0)$ satisfies the compatibilty condition 
\[
\int_{\Omega_0} ( \nabla u_0 | \nabla \varphi ) + \int_{S}|F_0'| (u_0 \varphi)(F_0) = \int_{S} |F_0'| H_{F_0} \varphi(F_0), \qquad \varphi \in W_2^1(\Omega_0), 
\]
the extended Stefan problem hat at least one mild solution 
\[
\rho \in C^{1/32}([0,T], W_2^2(S)), \qquad \qquad \tilde u_{F(\rho)} \in L_\infty([0,T] \rightarrow W_2^1(\tilde \Omega_F)).
\]
If, in addition, $H_{F_0} - \mbox{\normalfont tr }u_0 < 0$ in open subsets of $\mathcal{A}_0^\pm$, this solution develops self-intersections.
\item[2)] There exist $\varepsilon > 0$ and $\rho_0 \in \hat{\mathbf{E}}^*_\varepsilon$ (i.e. $\hat F_0 := F(\rho_0) = F_0 - \rho_0 N_{F_0}$ is an embedding) s.t. the same holds true, provided $u_0 \in W_2^1(\Omega_0)$ satisfies the compatibilty condition 
\begin{eqnarray*}
& & \int_{\Omega_0} |\mbox{\rm det} D \Theta_{\hat F_0} | ( \nabla u_0 | \nabla \varphi )_{\hat F_0} + \int_{S}|D \Theta_{\hat F_0}(F_0) F_0'| (u_0 \varphi)(F_0) \\ 
& = & \int_{S} |D \Theta_{\hat F_0}(F_0) F_0'| H_{\hat F_0} \varphi(F_0), \qquad \varphi \in W_2^1(\Omega_0).
\end{eqnarray*}
In particular, the corresponding solution $\rho_\varepsilon$ exists on the same interval $[0,T]$. (Clearly, at the moment $t_0$ a touching point occurs in $\tilde \gamma_{F(\rho(t_0,\cdot))}$, the Stefan problem is no longer well-posed in the classical sense.) Standard theory implies that prior to $t_0$, $\rho_\varepsilon$, $u_{F(\rho_\varepsilon)}$ are $C^\infty$ - functions of space and time, and that $(\rho_\varepsilon, u_{F(\rho_\varepsilon)})$ is a classical solution that is also unique.
\end{itemize}
\end{theorem}

\vspace{5mm}

\section{A suitable initial state}

\vspace{5mm}

\begin{minipage}{0.5\textwidth}
We start with a domain $\Omega$ with boundary curve $\Gamma$ oriented as in the picture on the right: The circular parts on the left have radii $R_1, R_2 \sim R >>1$, the circular parts on the right (with center $(x_0,\pm 1)$) have radius $1$. Parameterize (first) all parts by arclength to ensure that the curve is $C^1$.
\end{minipage}
\begin{minipage}{0.2\textwidth}
\[
\;
\] 
\end{minipage}
\begin{minipage}{0.3 \textwidth}
\includegraphics[width=6cm,height=7cm]{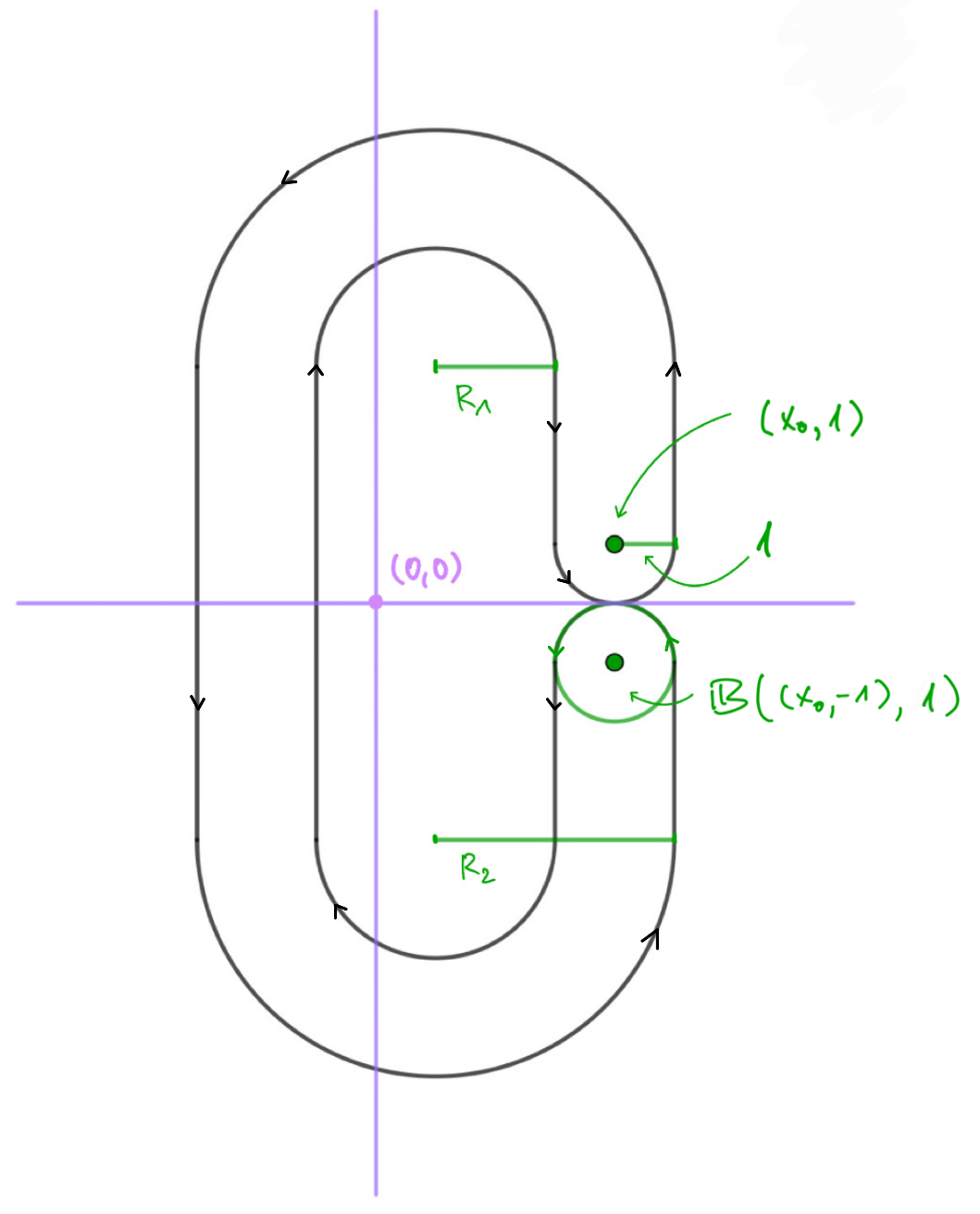}
\end{minipage} 

\vspace{10mm}  

\begin{minipage}{0.3\textwidth}
\includegraphics[width=3.7cm,height=5.2cm]{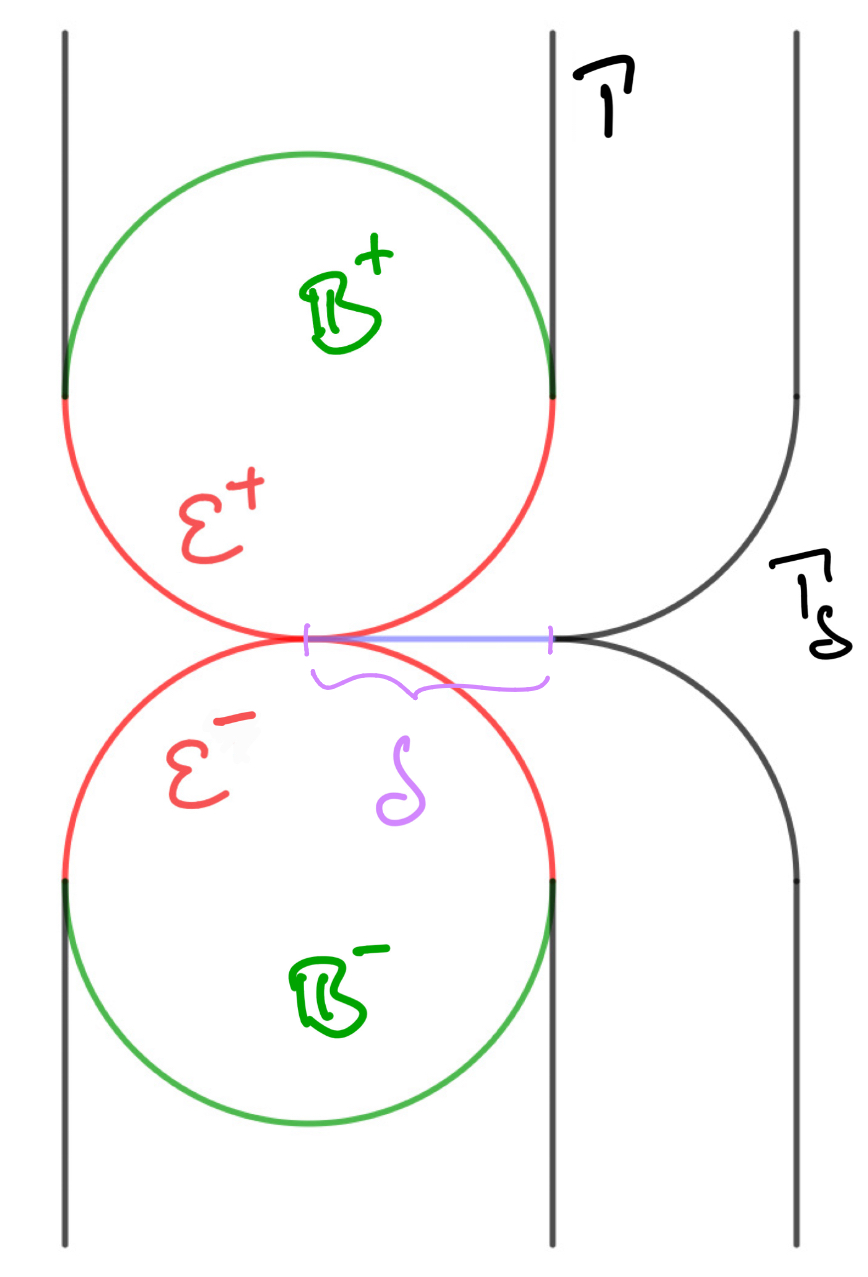}
\end{minipage} 
\begin{minipage}{0.2\textwidth}
\[
\;
\] 
\end{minipage}
\begin{minipage}{0.5\textwidth}
Next we introduce small flat parts of length $\delta$ in both  balls on the right (which we call $\mathbb{B}^\pm$) at the touching point $(x_0,0)$ (parameterized by arclength) s.t. we get a common flat segment. Now, reparameterize both curves ($\Gamma$ and $\Gamma_\delta$) over $S$.
\end{minipage} 

\vspace{10mm}  

Let $H$ and $H_\delta$ be the respective piecewise constant curvature functions. Arclength parameterizations in $C^1(S) \cap W_\infty^2(S)$ are then given by 
\[
G(s) := G(s_0) + \int_{s_0}^s \left( \cos(\theta(\sigma)), \sin(\theta(\sigma))  \right) \; d\sigma, \quad \qquad \theta(s) := \theta_0 + \int_{s_0}^s J(\sigma) \; d\sigma, 
\]
where $G \in \{ F, F_\delta \}$, $J \in \{ H, H_\delta \}$ and $G(s_0)$, $\theta_0$ are chosen suitably. From this formula it is clear that (since $H_\delta \rightarrow H$ in $L_2(S)$)
\begin{equation}\label{conFd}
F_\delta \rightarrow F \qquad \qquad \mbox{ in } C^1(S) \cap W_2^2(S) \qquad \qquad (\delta \rightarrow 0).
\end{equation}
We apply the following smoothing procedure:

\vspace{5mm}

We smooth the curvature of $\Gamma$, $\Gamma_\delta$ by means of the standard mollifier. Since the curves we discussed are in $W_2^\infty$ and have bounded (by $+1$ at the half circles on the right and $\sim \pm 1/R$ at the inner half circles on the left) curvature, the smoothing process converges in $L_p(S)$ ($p \geq 1$) and leaves the constant-curvature-parts uneffected modulo arbitrarily small errors.  However, the resulting curves (uniquely determined except for translation and rotation, which both is assumed to be performed suitably) may not be closed any more. The line segments and large circles on the left can be suitably shortened / extended to preserve the desired shapes. Since derivatives of a convolution can be placed to act on the mollifier, 
\begin{equation}\label{aphdelta}
\sup_{0 < \delta < \delta_0} \Vert H_\delta \Vert_{C^k(S)} + \Vert H \Vert_{C^k(S)} \leq M_k, \qquad \qquad k \in \mathbb{N},
\end{equation}
and hence also 
\begin{equation}\label{apfdelta}
\sup_{0 < \delta < \delta_0} \Vert F_\delta \Vert_{C^k(S)} + \Vert F \Vert_{C^k(S)} \leq M_k, \qquad \qquad k \in \mathbb{N}.
\end{equation}
By the young inequality, 
\begin{eqnarray*}\label{con2}
\Vert H_\delta \ast \varphi_\zeta - H \ast \varphi_\zeta \Vert_{L_2(S)} & = & \Vert (H_\delta - H) \ast \varphi_\zeta \Vert_{L_2(S)} \nonumber \\ 
& \leq & \Vert H_\delta - H \Vert_{L_2(S)} \Vert \varphi_\zeta \Vert_{L_1(S)} \nonumber \\
& = & \Vert H_\delta - H \Vert_{L_2(S)}. 
\end{eqnarray*}
From this and the fact that the 'length of the closedness preserving fills / cuts' goes to zero as $\zeta \rightarrow 0$ one derives that (\ref{conFd}) survives the smoothing procedure.

\vspace{5mm}

In the next section, we focus on the 'smaller' domains / curves without the '$\delta$-segments'.

\vspace{5mm}

The smoothed curve $\Gamma^{(s)}$ (again denoted by $\Gamma$ in the following) still contains two almost half-circles with radius 1, denoted by $\mathcal{E}^\pm$. Since convolution with a mollifier preserves local monotonicity of the curvature, we have that $\mathbb{B}^\pm = \mathbb{B}((x_0,\pm1),1) \subset \Omega$, $\overline{\mathbb{B}^\pm} \subset \overline{\Omega}$. (We denote the domain $\Omega^{(s)}$ inside $\Gamma^{(s)}$ again by $\Omega$.) 

\vspace{5mm}

We now introduce the parameter $\varepsilon > 0$ by $\varepsilon := 1/R$ and change notation from $\{ \Omega, \Gamma, F, \mathcal{E}^\pm \}$ to $\{ \Omega_\varepsilon, \Gamma_\varepsilon, F_\varepsilon, \mathcal{E}_\varepsilon^\pm \}$ in order to emphazise that we are going to leave the circular parts on the right with radius $1$ fixed but will enlarge the circular parts on the left with radius $\sim 1/\varepsilon$ suitably. Let $u_\varepsilon$ be defined by  
\[
\int_{\Omega_\varepsilon} ( \nabla u_\varepsilon | \nabla \varphi ) + \int_{S}|F_\varepsilon'| (u_\varepsilon \varphi)(F_\varepsilon) = \int_{S} |F_\varepsilon'| H_{F_\varepsilon} \varphi(F_\varepsilon), \qquad \varphi \in W_2^1(\Omega_\varepsilon).  
\]
Since $\Gamma_\varepsilon$ is smooth and elliptic regularity is a local property, $u_\varepsilon \in BUC^\infty (\Omega_\varepsilon,\RRM)$. Observe that 
\begin{equation}\label{apis}
\sup_{\varepsilon \in (0,\varepsilon_0)} \left( \Vert u_\varepsilon \Vert_{C(\overline{\Omega}_\varepsilon)} + \Vert u_\varepsilon \Vert_{W_2^1(\mathbb{B}^\pm)} + \Vert u_\varepsilon \Vert_{W_2^{1/2}(\partial \mathbb{B}^\pm)} \right)< \infty, 
\end{equation}
since 
\[
\int_{\Omega_\varepsilon} |\nabla u_\varepsilon|^2 + \frac{1}{2} \int_{S}|F_\varepsilon'| u_\varepsilon^2(F_\varepsilon) \leq \frac{1}{2} \int_{S} |F_\varepsilon'| H_{F_\varepsilon}^2 \leq \pi + \mathcal{O}(\varepsilon), 
\]
((\ref{apis}) also uses the standard trace result on $\mathbb{B}^\pm$) and in fact 
\begin{equation}\label{apis2}
- \varepsilon \leq u_\varepsilon \leq 1.
\end{equation}
(This can be deduced from (the proofs of) the elliptic maximum principle and the boundary point lemma.) 

\vspace{5mm}

We focus now on the solution $u_\varepsilon$ in $\overline{\mathbb{B}^+}$ (the analogue arguments hold for $\overline{\mathbb{B}^-}$) and write 
\[
u_\varepsilon^+ := u_\varepsilon |_{\overline{\mathbb{B}^+}} \qquad \qquad u_\varepsilon^+(x_0,0) := \lim_{h \downarrow 0} u_\varepsilon(x_0,h).
\]
Observe that, by smoothness, 
\[
\Delta u_\varepsilon^+ = 0 \qquad \mbox{ in $\overline{\mathbb{B}^+}$}, \qquad \qquad u_\varepsilon - \partial_{(\Gamma_\varepsilon)_{N_\varepsilon}} u_\varepsilon = 1 \qquad \mbox{ on $\mathcal{E}_\varepsilon^+$}.
\] 
($N_\varepsilon$ again is inner normal w.r.t. to $\mathcal{E}_\varepsilon^+$.) Our goal is to show that 
\[
u_\varepsilon^+(x_0,0) > 0.
\]
For simplicity, we drop the plus in our notation and write $u_\varepsilon$ instead of $u_\varepsilon^+$. Choose a cut-off function $\chi \in C^{\infty}(\overline{\mathbb{B}^+})$ that satisfies 

\vspace{5mm} 

\begin{minipage}{0.55\textwidth}
\begin{itemize}
\item $\chi = 1$ on a true subset $E^+ \subset \mathcal{E}^+$; 
\item $\chi \geq 0$ on $\overline{\mathbb{B}^+}$; $\chi > 0$ on $E^+ \cup \mathbb{B}^+$;
\item $\chi = 0$ on a true superset of $\partial \mathbb{B}^+ \setminus \mathcal{E}^+$; 
\item $0 \leq \chi \leq 1$, $\Delta \chi = 0$ in $\mathbb{B}^+$. 
\end{itemize}
\end{minipage}
\begin{minipage}{0.15\textwidth}
\[
\;
\] 
\end{minipage}
\begin{minipage}{0.3 \textwidth}
\includegraphics[width=3cm,height=4cm]{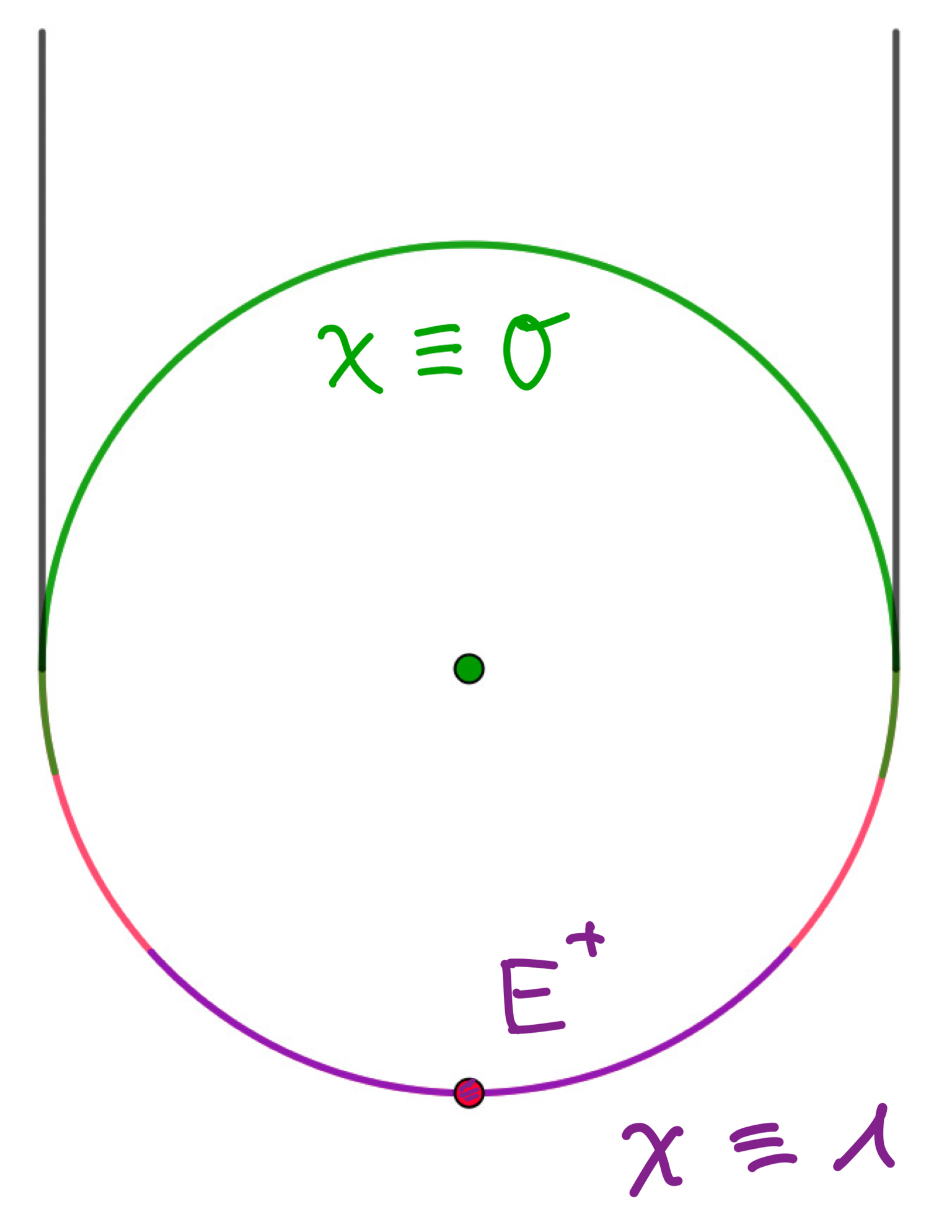}
\end{minipage}

\vspace{5mm} 

($\chi$ is obtained by solving a suitable Dirichlet Problem). Then for $z_\varepsilon := u_\varepsilon \chi$ we have 
\[
\Delta z_\varepsilon = 2 (\nabla u_\varepsilon | \nabla\chi) \qquad \mbox{ in } \mathbb{B}^+, \qquad \qquad z_\varepsilon - \partial_N z_\varepsilon = \chi - u_\varepsilon \partial_N \chi \qquad \mbox{ on } \partial \mathbb{B}^+. 
\]
By (\ref{apis}) and elliptic regularity, 
\[
\sup_{\varepsilon \in (0,\varepsilon_0)} \Vert z_\varepsilon \Vert_{W_2^2(\mathbb{B}^+)} < \infty
\]
and 
\[
u_\varepsilon \rightharpoonup u \quad \mbox{ in } W_2^1(\mathbb{B}^+), \qquad \qquad \qquad z_\varepsilon \rightharpoonup z \quad \mbox{ in } W_2^2(\mathbb{B}^+).
\]
Since $z \in W_2^2(\mathbb{B}^+) \hookrightarrow C(\overline{\mathbb{B}^+})$, $u_\varepsilon, z_\varepsilon \geq - \varepsilon$, $z \geq 0$: 

\vspace{5mm}

Indeed, suppose that $z < 0$ in $a \in \mathbb{B}^+$. Then $z \leq - \nu$ in some ball $B(a,2r) \subset \mathbb{B}^+$ ($\nu  > 0$). Pick $\varphi \in C^\infty(\overline{\mathbb{B}^+})$, $0 \leq \varphi \leq 1$, s.t. $\varphi = 1$ in $B(a,r)$ and $\varphi = 0$ in $\overline{\mathbb{B}^+} \setminus B(a,2r)$. Then  
\[
- 4 \varepsilon \pi r^2 \leq \int_{\mathbb{B}^+} z_\varepsilon \varphi \sim \int_{\mathbb{B}^+} z \varphi = \int_{B(a,2r)} z \varphi \leq - \nu \pi r^2 \qquad \qquad (\varepsilon << \nu / 4), 
\]
hence $\varepsilon \geq \nu/4$, a contradiction.

\vspace{5mm}

Since $u = z / \chi$ where $\chi > 0$, $u \in C(\{ \chi > 0\})$ and $z = u$ on $E^+$ (in the sense of a continuous representative). In particular, there exists a ball $B \subset \RRM^2$, $E^+ \subset B$ s.t. $u \in C \left( \overline{B \cap \mathbb{B}^+} \right) $, and $z \in W_2^2(\mathbb{B}^+)$ is the weak solution of 
\[
\Delta z = 2 (\nabla u | \nabla\chi) \qquad \mbox{ in } \mathbb{B}^+, \qquad \qquad z - \partial_N z = \chi - u \partial_N \chi \qquad \mbox{ on } \partial \mathbb{B}^+.
\]
Observe that (by zero extension), $f := 2 (\nabla u | \nabla\chi)$ can be considered an element of $L_2(\mathbb{R}^n)$. Let $f_\delta := f \ast \varphi_\delta$, $u_\delta := u \ast \varphi_\delta$ be the convolutions with the standard mollifier. Then 
\[
f_\delta \rightarrow f \quad \mbox{ in } L_2(\mathbb{B}^+) 
\]
and  
\[
u_\delta \rightarrow u \quad \mbox{ in } W_2^1(\mathbb{B}^+) \cap W_2^{1/2}(\partial \mathbb{B}^+) \cap C \left( \overline{B \cap \mathbb{B}^+} \right)
\]
(using also the Tietze extension theorem). Let $y_\delta \in C^\infty(\overline{\mathbb{B}^+})$ solve 
\[
\Delta y_\delta = f_\delta \qquad \mbox{ in } \mathbb{B}^+, \qquad \qquad y_\delta - \partial_N y_\delta = \chi - u_\delta \partial_N \chi \qquad \mbox{ on } \partial \mathbb{B}^+.
\]
Since the $y_\delta$ are uniformly bounded in $W_2^2(\mathbb{B}^+)$ and $y_\delta \rightarrow z$ in $W_2^1(\mathbb{B}^+)$, interpolation yields that 
\[
y_\delta \rightarrow z \quad \mbox{ in } W_2^{2-\nu}(\mathbb{B}^+) \hookrightarrow C(\overline{\mathbb{B}^+}) \qquad \qquad (0 < \nu < 1).
\]

\vspace{5mm}

Assume now that $u(x_0,0) = 0$. Then $z(x_0,0) = 0$ and thus $y_\delta(x_0) \leq 1/8$, hence $\partial_N y_\delta(x_0,0) \leq - 3/4$ (for small enough $\delta$ s.t. $|(u_\delta \partial_N \chi)(x_0,0)| \leq 1/8$). Since 
\begin{eqnarray*}
\partial_N y_\delta(x_0,0) & \sim & \frac{1}{h} \; \left( y_\delta((x_0,0) + h N) - y_\delta(x_0,0) \right) \\
& \sim & \frac{1}{h} \; \left( z((x_0,0) + h N) - z(x_0,0) \right) \\
& & (0 < \delta, h << 1),   
\end{eqnarray*} 
we have 
\[
z((x_0,0) + h N) < - \frac{h}{2}+ z(x_0,0) < 0, 
\]
which contradicts $z \geq 0$ in $\overline{\mathbb{B}^+}$. Hence, $u(x_0,0) > 0$ and therefore $u_\varepsilon > 0$, if $\varepsilon > 0$ is small enough:

\vspace{5mm}

Indeed, suppose that $u_\varepsilon(x_0,0) \leq 0$. Then $u \geq \nu$ and $u_\varepsilon \leq \nu/8$ in some set $E := B((x_0,0),2r) \cap \overline{\mathbb{B}^+}$ ($\nu  > 0$). Pick $\varphi \in C^\infty(\overline{\mathbb{B}^+})$, $0 \leq \varphi \leq 1$, s.t. $\varphi = 1$ in $B((x_0,0),r) \cap \overline{\mathbb{B}^+} \subset E$ and $\varphi = 0$ in $E^c = \overline{\mathbb{B}^+} \setminus B((x_0,0),2r)$. We have  
\[
\pi r^2 \nu / 2 \geq \int_{E} u_\varepsilon \varphi = \int_{\mathbb{B}^+} u_\varepsilon \varphi \sim \int_{\mathbb{B}^+} u \varphi = \int_{E} u \varphi \geq \nu \pi r^2, 
\]
a contradiction. At this point we fix a suitable $\varepsilon$ and drop it from our notation, i.e. write $u = u^+$ instead of $u_\varepsilon$ and store the fact that $u^+(x_0,0) > 0$. (Analogue arguments yield that also $u^-(x_0,0) > 0$, where $u^-(x_0,0)$ is defined in the obvious way.) By continuity, we can fix $\nu > 0$ s.t. $u^\pm > 0$ on $\mathbb{B}(x_0,\nu) \cap \overline{\Omega}^\pm$. ($\Omega^\pm$ $\Omega_\delta^\pm$, $\Gamma^\pm$, $\Gamma_\delta^\pm$ are defined in analogy to section 2.1)

\vspace{5mm}

Finally, we come back to the  enlarged domains $\Omega_\delta$ and aim to show that the corresponding solutions $u_\delta$ ( i.e. 
\[
\int_{\Omega_\delta} ( \nabla u_\delta | \nabla \varphi ) + \int_{S}|F_\delta'| (u_\delta \varphi)(F_\delta) = \int_{S} |F_\delta'| H_{F_\delta} \varphi(F_\delta), \qquad \varphi \in W_2^1(\Omega_\delta) \quad \mbox{ ) }  
\]
will still be positive at parts of the flat segments (if $\delta > 0$ is small enough). Observe that $\Vert F_{\delta}\Vert_{C^k(S)} \leq M_k$ for $k \in \mathbb{N}$. Localization, the classcical halfspace result for Robin type boundary conditions coupled with elliptic operators and (\ref{apfdelta}) yield $u_\delta \in BUC^\infty(\Omega)$ and 
\[
\sup_{0 < \delta < \delta_0} \Vert u_\delta \Vert_{BUC^{k}(\Omega)} \leq C_k \qquad \qquad (k \in \mathbb{N}) 
\]
as well as 
\[
\Delta u_\delta = 0 \quad \mbox{ in } \Omega_\delta \supset \Omega, \qquad \qquad (u_\delta - (\nabla u_\delta | N_{F_\delta})) \circ F_\delta = H_\delta \quad \mbox{ on } S.
\]
By compact embedding, along some subnet $\tilde \delta$, again denoted by $\delta$,  
\[
u_\delta |_{\Omega^\pm} \rightarrow w^\pm \qquad \mbox{ in } C^2 \left( \overline{\Omega^\pm} \right).
\]
(We need to be a bit careful here, since the compact embedding argument uses Arzel\`{a} - Ascoli and hence evaluation in the double point.) However, since $w^+ = w^-$ in $\overline{\Omega^+} \cap \overline{\Omega^-} \cap \{ (x,y); \; x < 0 \}$ by uniqueness of the limits, and $w^+$, $w^-$ naturally give rise to a function $w \in BUC^2(\Omega)$ that satisfies 
\[
u_\delta \rightarrow w \quad \mbox{ in } BUC^2(\Omega), \qquad \qquad \Delta w = 0 \quad \mbox{ in } \Omega.
\]
Moreover, 
\begin{eqnarray*}
|u_\delta(F_\delta(s)) - w(F(s))| & \leq & |u_\delta(F_\delta(s)) - u_\delta(F(s))| + |u_\delta(F(s)) - w(F(s))| \\
& \leq & C_1 |F_\delta(s)) - F(s)| + |u_\delta(F(s)) - w(F(s))| \\
& & \rightarrow 0 \qquad \qquad (s \in S).
\end{eqnarray*}
The analogue argument (invoking (\ref{conFd})) implies $ (\nabla u_\delta | N_{F_\delta}) \circ F_\delta \rightarrow (\nabla w | N_{F}) \circ F $, and we find
\[
(w - (\nabla w | N_{F})) \circ F = H \quad \mbox{ a.e. on } S, 
\]
which implies $w = u$. Hencey, $u_\delta^\pm > 0$ on, say,  $\mathbb{B}(x_0,\nu/2) \cap \overline{\Omega^\pm}$.

\vspace{5mm}

\section{Appendix}

\vspace{5mm}

\begin{lemma}
For $u,v \in W_2^1(\Omega_0)$ we have 
\[
\int_{\Omega_0} u \partial_j v + \int_{\Omega_0} \partial_j u v = - \int_{S} (u v)(F_0) n_{0,j} |F_0'|.
\]
\end{lemma}
\begin{proof}
Let $u,v \in W_2^1(\Omega_0)$. 

\begin{minipage}{0.5\textwidth}
\begin{flushleft}
By Theorem 2.1.4 in \cite{Zie} there is a horizontal line $\lambda \subset \Omega_0^+ \cap \Omega_0^-$ that splits $\Omega_0$ into two Lipschitz domains $D_0^+$ and $D_0^-$, and $u,v$ are absolutely continuous along this line. W.l.o.g., $\partial D_0^\pm \setminus \lambda \subset \gamma_0^\pm = F_0[X^\pm]$.
We denote by $N_{\gamma_0}^\pm = (n_{{\gamma_0},1}^\pm, n_{{\gamma_0},2}^\pm)^T$ the respective inner unit normal fields on $\partial D_0^\pm$ (set as usual $N_{0}^\pm := N_{\gamma_0}^\pm \circ F_0 |_{X^\pm}$) and calculate 
\end{flushleft}
\end{minipage}
\begin{minipage}{0.1\textwidth}
$\;$
\end{minipage}
\begin{minipage}{0.4\textwidth}
\begin{flushleft}
\includegraphics[width=3.6cm,height=4.5cm]{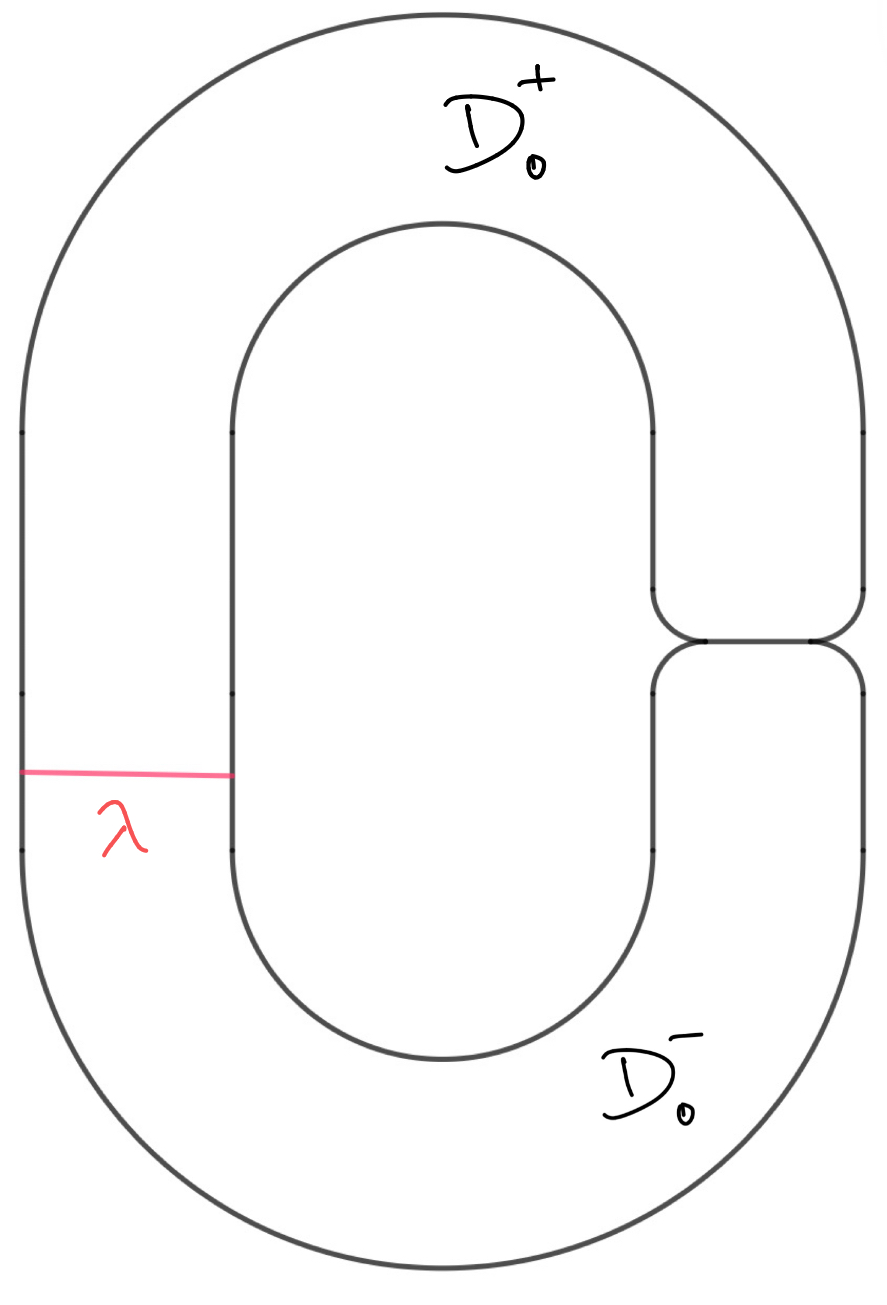}
\end{flushleft}
\end{minipage}

\vspace{5mm}

\begin{eqnarray*}
\int_{\Omega_0} u \partial_j v + \int_{\Omega_0} \partial_j u v & = & \int_{D_0^+} u \partial_j v + \partial_j u v + \int_{D_0^-} u \partial_j v + \partial_j u v \\
& = & - \int_{\partial D_0^+ \setminus \lambda} u v n_{\gamma_0,j}^+ - \int_{\partial D_0^- \setminus \lambda} u v n_{\gamma_0,j}^- \\
&    & - \int_\lambda u v n_{\gamma_0,j}^+ - \int_\lambda u v n_{\gamma_0,j}^- \\
& = & - \int_{S} (u v)(F_0) n_{0,j} |F_0'| \\
&    & + \quad (j-1) \cdot \left\{ - \int_\lambda u v + \int_\lambda u v \right\} \\
& = &  - \int_{S} (u v)(F_0) n_{0,j} |F_0'|, 
\end{eqnarray*}
$j \in \{ 1, 2\}$.
\end{proof}

\vspace{5mm}

\begin{lemma}
The forms $\tilde a_F$ in (\ref{eSpB}) are uniformly coercive w.r.t. $F \in \mathbf{B}_\varepsilon$.
\end{lemma}
\begin{proof}
Because of the decomposition
\begin{eqnarray*}
\tilde a_F(u,u) & = & \int_{\Omega_0} |1 - (1 - \mbox{det} D \tilde \Theta_F)| (\nabla u | \nabla u)_{\tilde F} + \int_S |D (\tilde \theta_F \circ F_0)| u^2 \\
& \geq &  \int_{\Omega_0} (\nabla u | \nabla u)_{\tilde F} - \int_{\Omega_0} |1 - \mbox{det} D \tilde \Theta_F| (\nabla u | \nabla u)_{\tilde F} + \int_S |D ( \tilde \theta_F \circ F_0)| u^2, \\
\end{eqnarray*}
\begin{eqnarray*}
\left( \nabla u | \nabla u \right)_{\tilde F} & = & \left( \; \left( \left( D \tilde \Theta_F^{-1}\right)^T (\tilde \Theta_F) - \mathbf{Id} \right) \nabla u \; | \;  \left( D \tilde \Theta_F^{-1} \right)^T (\tilde \Theta_F) \nabla u \; \right) \\
& + & \left( \; \nabla u \; | \;  \left( \left( D \tilde \Theta_F^{-1}\right)^T (\tilde \Theta_F) - \mathbf{Id} \right) \nabla u \; \right) \\
& + & \left( \; \nabla u \; | \;  \nabla u \; \right) \\
\end{eqnarray*}
and 
\begin{eqnarray*}
\int_S |D (\tilde \theta_F \circ F_0)| u^2 & = & \int_S |F_0' - ( F_0' - D ( \tilde \theta_F \circ F_0) )| u^2 \\ \\
& \geq &  \int_S |F_0'| u^2 - \int_S |F_0' - D ( \tilde \theta_F \circ F_0)| u^2, 
\end{eqnarray*}
in view of (\ref{stheta}), (\ref{D1}) - (\ref{D4}), it suffices to show that 
\[
a(u,v) := \int_{\Omega_0} (\nabla u | \nabla v) + \int_S |F_0'| (u v)(F_0)
\]
is coercive. (Observe that w.l.o.g. $|F_0'| > 0$ on S.) We can almost follow the standard argumentation: If not, there would be a sequence $(u_k) \subset W_2^1(\Omega_0)$, $\Vert u_k \Vert_{W_2^1(\Omega_0)} = 1$ s.t. 
\begin{itemize}
\item[(i)] $\int_{\Omega_0} |\nabla u_k|^2 \rightarrow 0$;
\item[(ii)] $\int_S |F_0'| u_k^2(F_0) \rightarrow 0$.
\end{itemize}
Let $D_0^+ := \Omega_0 \cap \{ (x,y); \; y > 0\}$, $D_0^- := \Omega_0 \cap \{ (x,y); \; y < 0\}$. Since the boundaries $\partial D_0^\pm$ are Lipschitz (and hence have the extension property for $W_p^1$ - functions), after passing to a subsequence (twice), the Rellich-Kondrachov Theorem yields the existence of $u^\pm \in L_2(D_0^\pm)$ s.t. 
\begin{itemize}
\item $\int_{D_0^\pm} |u_k - u^\pm|^2 \rightarrow 0$.
\end{itemize}
$u^+$ and $u^-$ naturally form an element $u \in L_2(\Omega_0)$ s.t. 
\begin{itemize}
\item[(iii)] $\int_{\Omega_0} |u_k - u|^2 = \int_{D^+_0} |u_k - u^+|^2 + \int_{D^-_0} |u_k - u^-|^2 \rightarrow 0$.
\end{itemize}
From the previous Lemma, (i), (ii) and (iii) we have 
\[
\int_{\Omega_0} u \partial_j v \leftarrow \int_{\Omega_0} u_k \partial_j v = - \int_{\Omega_0} \partial_j u_k v - \int_{S} (u_k v)(F_0) n_{0,j} |F_0'| \rightarrow 0, \qquad v \in W_2^1(\Omega_0), 
\]
and hence $u \in W_2^1(\Omega_0)$ with $\nabla u = 0$ a.e., which implies that $u$ is constant a.e. (and $u_k \rightarrow u$ in $W_2^1(\Omega_0)$). Since the trace mapping from section 2.1 is continuous, by (ii), $0 \leftarrow \int_S |F_0'| u_k^2(F_0) \rightarrow \int_S |F_0'| u^2(F_0)$, which implies $u \equiv 0$ a.e., contradicting $\Vert u_k \Vert_{W_2^1(\Omega_0)} = 1$ for all $k$.
\end{proof}
\begin{lemma}\label{contws}
Let $H$ be a Hilbert space over $\KKM$ and let $\mathcal{S}^2(H)$ be the Banach space of bounded sesquilinear forms $a: H \times H \rightarrow \KKM$ with norm
\[
\Vert a \Vert_{\mathcal{S}^2(H)} := \sup \{ |a(u,v)|; \; \Vert (u,v) \Vert_{H \times H} \leq 1 \}.
\]
Let $C, c > 0$ and $a \in Y_{C,c} \subset \mathcal{S}^2(H)$ iff 
\begin{itemize}
\item $\forall \; (u,v) \in H \times H$: $|a(u,v)| \leq C \Vert u \Vert_H \Vert v \Vert_H$;
\item $\forall \; u \in H$: $\Re a(u,u) \geq c \Vert u \Vert^2$.
\end{itemize}
For $(a,b) \in Y_{C,c} \times H'$ denote by $w(a,b)$ the unique solution of the equation 
\[
a(v,w) = b(v), \qquad \qquad v \in H 
\]
(Lax-Milgram, Riesz). The mapping 
\[
(a,b) \mapsto w(a,b): Y_{C,c} \times H' \rightarrow H
\]
is locally Lipschitz continuous.
\end{lemma}
\begin{proof}
Let $L: Y_{C,c} \times H \rightarrow H'$ be defined by $L(a,v) := a(\cdot,v)$, and let $J: H \rightarrow H'$ be the Riesz isometry. Then 
\[
w(a,b) = (J^{-1} L(a,\cdot))^{-1}J^{-1}(b).
\]
Since inversion a smooth operation on $\mathcal{L}_{is}(H)$, the assertion follows from the fact that $L$ is Lipschitz continuous on bounded sets.  
\end{proof}

\end{document}